\documentclass[12pt,a4paper,oneside]{amsart}
\usepackage{a4,a4wide}
\usepackage{amsfonts,amsmath,amssymb,amsthm,enumitem}

\usepackage{mathtools}

\usepackage{color}
\usepackage{ifpdf}
\ifpdf
    \usepackage[pdftex]{graphicx}
    \DeclareGraphicsExtensions{.png,.pdf,.jpg}
\else
   \usepackage[dvips]{graphicx}
   \DeclareGraphicsExtensions{.eps}
\fi
\graphicspath{{.}{figures/}}
\usepackage[latin1]{inputenc}
\numberwithin{equation}{section}

\usepackage{doi}

\newtheorem{theorem}{Theorem}[section]
\newtheorem{lemma}[theorem]{Lemma}
\newtheorem{proposition}[theorem]{Proposition}

\newtheorem{remark}[theorem]{Remark}
\newtheorem*{theorem*}{Theorem}
\newtheorem*{lemma*}{Lemma}
\newtheorem*{proposition*}{Proposition}
\newtheorem*{corollary*}{Corollary}

\renewcommand\tilde{\widetilde}

\def\R{\mathbb{R}}

\def\Z{\mathbb{Z}}
\def\N{\mathbb{N}}
\def\EE{\mathbb{E}}
\def\PP{\mathbb{P}}

\def\P{\mathbb{P}}

\def\LM#1{\hbox{\vrule width.2pt \vbox to#1pt{\vfill \hrule width#1pt
height.2pt}}}
\def\LL{{\mathchoice {\>\LM7\>}{\>\LM7\>}{\,\LM5\,}{\,\LM{3.35}\,}}}
\def\restr{{\LL}}
\renewcommand{\phi}{\varphi}

\def\1{\mathbf{1}}
\def\loc{\mathrm{loc}}

\def\XXint#1#2#3{{\setbox0=\hbox{$#1{#2#3}{\int}$ }
\vcenter{\hbox{$#2#3$ }}\kern-.57\wd0}}

\def\eps{\varepsilon}
\def\lt{\left}
\def\rt{\right}
\def\les{\lesssim}
\def\ges{\gtrsim}

\def\spt{\textup{Spt}\,}

\newcommand{\bra}[1]{\left( #1 \right)}
\newcommand{\sqa}[1]{\left[ #1 \right]}

\newcommand{\nor}[1]{\left\| #1 \right\|}

\begin{document}
\title[A fluctuation result for the displacement in the optimal matching problem]{
A fluctuation result for the displacement in the optimal matching problem}
\author{Michael Goldman \address[Michael Goldman]{Universit\'e de Paris, CNRS,   Laboratoire Jacques-Louis Lions (LJLL),  France} \email{michael.goldman@u-paris.fr} \hspace*{0.5cm} Martin Huesmann \address[Martin Huesmann]{Universit\"at M\"unster,  Germany} \email{martin.huesmann@uni-muenster.de}}
\thanks{We warmly thank F. Otto for numerous discussions at various stages of the project. We also thank N. Berestycki and J. Aru for suggesting the connection between the standard GFF and the curl-free variant we consider (see Remark \ref{rem:GFF}). MH is funded by the Deutsche Forschungsgemeinschaft (DFG, German Research Foundation) under Germany's Excellence Strategy EXC 2044 -390685587, Mathematics M\"unster: Dynamics--Geometry--Structure and   by the DFG through the SPP 2265 {\it Random Geometric Systems. MG is partially funded by the ANR project SHAPO  }}

\begin{abstract}
 The aim of this paper is to justify in dimensions two and three the ansatz of Caracciolo et al.\ stating that  the displacement in the optimal matching problem is
 essentially given by the solution to the linearized equation i.e.\ the Poisson equation. Moreover, we prove that at all mesoscopic scales, this displacement is
 close in suitable negative Sobolev spaces to a curl-free Gaussian free field.
 For this we combine a quantitative estimate on the  difference between the displacement and the linearized object, which is based on the large-scale regularity theory recently developed in collaboration with F. Otto, together with a qualitative convergence result for the linearized problem.
\end{abstract}

\date{\today}
\maketitle

\section{Introduction}

The optimal matching problem is one of the classical random optimization problems which has received constant attention in the probability and mathematical physics literature over the last 30 years, e.g.\ \cite{AKT84, Ta94, Yu98, Ta14, FoGu15, DeScSc13, CaSi14, CaLuPaSi14, AmStTr16, AmGlaTre, Le17, Le18, BobLe}. We are interested in one of its simplest and most studied variants: Let $X_1,X_2,\ldots$ be iid uniformly distributed random variables on the torus $Q_L=[-\frac{L}{2},\frac{L}{2})^d$ and put $$\mu^{R,L}=\frac1{R^d} \sum_{i=1}^{(RL)^d} \delta_{X_i}.$$ Notice that the typical distance between points is  $1/R$. The optimal matching problem is then
\begin{equation}\label{eq:matching}
\mathsf{C}_{R,p,d}=\EE\lt[\inf_{\pi\in\mathsf{Cpl}(\mu^{R,L},\mathsf{Leb}\restr Q_L)}\int_{Q_L\times Q_L} |x-y|^p d\pi\rt],
\end{equation}
where $\mathsf{Cpl}(\mu,\nu)$ denotes the set of all couplings between $\mu$ and $\nu$. By now, the asymptotic behaviour of the expected cost $\mathsf{C}_{R,p,d}$  is well understood, e.g.\ \cite{AKT84, BaBo13,DeScSc13, AmStTr16, GolTre}, see also \cite{FoGu15} for deviation estimates. 
 We refer to  \cite{Le19} for a fluctuation result of the transport cost in $H^{-1}$.
In dimension one, the behaviour of the optimal transport cost is very well understood also beyond the case of uniformly distributed points. The main reason is that the optimal couplings --the minimizers of \eqref{eq:matching}-- are explicit, see \cite{BoLe19}.

Turning to the asymptotic behaviour of the optimal coupling itself, respectively the displacement under the optimal coupling, not much is known, see \cite{HuSt13} for a global existence result on invariant couplings between the Lebesgue measure and a Poisson point process.
In this article, we focus on the quadratic case $p=2$ and show that at all mesoscopic scales the averaged displacement converges to a Gaussian field. At the macroscopic scale, a related result has been obtained in \cite{AmGlaTre}. To explain our result we need to introduce some notation.

We fix $p=2$ and consider the $Q_L$-periodic version of \eqref{eq:matching}. Denote the $Q_L$-periodic optimal coupling between $\mu^{R,L}$ and $\mathsf{Leb}\restr Q_L$ by $\pi^{R,L}.$
We define the distribution $Z^{R,L}$ by
\[
 Z^{R,L}(f)=R^{\frac{d}{2}}\int_{Q_L\times \R^d} f(x) (y-x) d\pi^{R,L}.
\]
Notice  that if $T$ is the optimal transport map from $\mathsf{Leb}$ to $\mu^{R,L}$, i.e.\ $\pi^{R,L}=(T,Id)_\#(\mathsf{Leb}\restr Q_L)$, and if $A_i=\{y : T(y)=X_i\}$ are the corresponding Laguerre cells then 
\begin{equation}\label{Z}
 Z^{R,L}=R^{\frac{d}{2}}\sum_{i=1}^{(RL)^d} \lt(\int_{A_i} (y-X_i) dy\rt)\delta_{X_i}. 
\end{equation}
For $d\geq 3$ and $f\in C_c^\infty(\R^d,\R^d)$ let $\phi^\infty$ be the unique solution in $L^2(\R^d)$ of\footnote{we use the notation $\nabla \cdot f$ for the divergence and $\Delta f=\nabla\cdot(\nabla f)$ for the Laplacian.} 
$$-\Delta \phi^\infty=\nabla\cdot f \qquad \text{in } \R^d.$$
We then define the curl-free Gaussian free field $\nabla \Psi$ (see Remark \ref{rem:GFF} for the connection with the standard Gaussian free field) by requiring
$$ \mathsf{Law}(\nabla \Psi(f)) =\mathcal N\lt(0,\int (\phi^\infty)^2\rt),$$
so that, denoting by $W$ white noise on $\R^d$, $\Psi$ formally solves
$$\Delta\Psi=W \qquad \text{in } \R^d.$$
For $\gamma>0,$ $ p\geq 1,$ we denote the local fractional Sobolev space by $W^{-\gamma,p}_{loc}$, see Section \ref{sec:negSob}.
We then have the following result:
\begin{theorem}\label{thm:intro1}
 For $d=3$, $p\geq 2,$ $\gamma> d\lt(1-\frac1p\rt),$ and any sequences $R,L \to \infty$ the distribution $ Z^{R,L}$ converges in law in $W^{-\gamma,p}_{loc}$ to $\nabla \Psi.$ Moreover, for every fixed $\ell\ge 1$, $Z^{R,L}$ has bounded moments of arbitrary order in $W^{-\gamma,p}(B_{\ell})$. 
\end{theorem}
This result is indeed about mesoscopic scales since, up to rescaling, the macroscopic scale corresponds to $L=1$, $R\to \infty$ while the microscopic scale is seen in the opposite regime $L\to \infty$ and $R=1$. Let us point out that up to the appropriate modification in the definition of the Gaussian field $\nabla \Psi$  our result could be easily extended to the macroscale (see also Remark \ref{rem:macroscale}). We leave the details to the reader.  The microscopic behavior is more subtle to analyze and of fundamentally different nature. We refer to \cite{GHO1} for preliminary results in this direction.\\
We expect Theorem \ref{thm:intro1} to hold also for $d>3$. However, central to our proof  are the deterministic quantitative estimates  from \cite{GHO} which are not precise enough to conclude in $d>3$. In $d=2$, Theorem \ref{thm:intro1} cannot hold since $\nabla \Psi$ does not exist. However, $\nabla \Psi(f)$ is well defined in $d=2$ for functions $f$ with $\int f=0$. Said differently, $\nabla \Psi$ is well defined only up to a constant. Hence, to obtain a version of Theorem \ref{thm:intro1} some kind of renormalization is needed. To this end, we fix a smooth radial  cutoff function $\eta\ge 0$ with $\int\eta=1.$ Define $\phi^\infty$ to be the $L^2(\R^2)$ solution of 
$$-\Delta \phi^\infty=\nabla \cdot \lt(f-\eta\int f \rt) \qquad \text{in } \R^d.$$
We then define $\nabla\Psi-\nabla\Psi_1(0)$ by requiring
$$\mathsf{Law}((\nabla \Psi-\nabla\Psi_1(0))(f))=\mathcal N\lt(0,\int (\phi^\infty)^2\rt).$$
To identify the correct renormalization on the level of the displacement $Z^{R,L}$ we introduce as an approximation of white noise
$$W^{R,L}=R^{\frac{d}{2}}\lt(\mu^{R,L}-1\rt)$$
and let $u^{R,L}$ be the $Q_L$-periodic  solution to $\Delta u^{R,L}=W^{R,L}$ with $\int_{Q_L} u^{R,L}=0.$ We put $\nabla u_1^{R,L}(0)=\int \eta \nabla u^{R,L}$. We then have

\begin{theorem}\label{thm:intro2}
 For $d=2$, $p\geq 2,$ $\gamma> d\lt(1-\frac1p\rt),$ and any sequences $R,L \to \infty$ the distribution $ Z^{R,L}-\mu^{R,L}\nabla u^{R,L}_1(0)$ converges in law in $W^{-\gamma,p}_{loc}$ to $\nabla \Psi-\nabla\Psi_1(0).$ Moreover, for every fixed $\ell\ge 1$, it has bounded moments of arbitrary order in $W^{-\gamma,p}(B_{\ell})$. 
\end{theorem}

Regarding the behavior of the logarithmically diverging shift $\nabla u_1^{R,L}(0)$, using the Green function representation of $u^{R,L}$ together with the quantitative CLT in  \cite{Bo20} we show in Lemma \ref{lem:nablauCLT} that 
 setting $\sigma^2 = \frac{1}{2} \EE[|\nabla u^{R,L}_1(0)|^2]$, we have  $\sigma^2\sim \log L$
and for any $p\geq 2$\footnote{We denote by $W_p$ the $p-$Wasserstein distance . The notation $A\ll 1$, which we only use in assumptions, means that there exists an $\eps>0$ only depending on the dimension, the fixed cutoff function $\eta$ and the parameters $p$ and $\gamma$, such that
if $A\leq \eps$ then the conclusion holds.
 Similarly, the notation $A\lesssim B$, which we use in output statements, means that there exists a global
constant $C>0$ depending on the dimension, the fixed cutoff function $\eta$ and the parameters $p$ and $\gamma$ such that $A\le C B$.}
$$W_p\lt(\mathsf{Law}\lt(\nabla u^{R,L}_1(0)\rt),\mathcal N\lt(0,\sigma^2 \mathsf{Id}\rt)\rt) \les \frac{1}{R \log^{\frac{1}{p}} L}.$$
Notice that  if we define $\nabla \Psi^L$ as the $Q_L-$periodic analog of the curl-free Gaussian free field $\nabla \Psi$, this may be equivalently written as 
\[
 W_p\lt(\mathsf{Law}\lt(\nabla u^{R,L}_1(0)\rt),\mathsf{Law}\lt(\nabla \Psi^{L}_1(0)\rt)\rt) \les \frac{1}{R \log^{\frac{1}{p}} L}.
\]

The proof of Theorems \ref{thm:intro1} and \ref{thm:intro2} will be done in two steps. In Section \ref{sec:linear}, we show the linear part, the distributional convergence of $\nabla u^{R,L}$ in $W^{-\gamma,p}_{loc}$ to $\nabla \Psi$ (with the appropriate renormalization when $d=2$). Let us point out that this result holds in any dimension and could be of independent interest. The main ingredient is to show tightness which follows from careful moment estimates together with an estimate of the negative Sobolev norm  through convolutions in Theorem \ref{theo:mainnegsob} which seems to be new. Then, we only have to identify the limit wich follows by mimicking the proof of the CLT.

The second step consists of two parts, carried out in Section \ref{sec:deter} and \ref{sec:sto}. First we need to adapt the results of \cite{GHO} to the present setup. This yields (see Theorem \ref{theo:main}) a quantitative deterministic estimate of the local average of the displacement under an optimal coupling between $\mu^{R,L}$ and $\mathsf{Leb}$ to the local average of the gradient of the solution to the Poisson equation 
$\Delta u^{R,L}=\mu^{R,L}-1.$
Secondly, we show that one can use the results of \cite{AmStTr16} and \cite{GolTre} to show that the assumptions for Theorem \ref{theo:main} are satisfied for a random radius $r_{*,L}$ with stretched exponential moments independent of $L$, see Theorem \ref{thm:mainstoch}. Combining these two parts, writing $\eta_r(x)=\frac{1}{r^d}\eta\lt(\frac{x}{r}\rt)$, we obtain for $1\leq r \ll L$, $p\geq 1$ the annealed estimate (see Proposition \ref{theop}) between the displacement under the optimal coupling and the solution to the linearized problem $\nabla u^{1,L}_r(0)=\int\eta_r(x)\nabla u^{1,L}(x)$
$$\EE\lt[\lt|\int \eta_r (x) (y-x- \nabla u^{1,L}_r(0)) d\pi^{1,L}\rt|^p\rt]^{\frac{1}{p}}\les \frac{\beta(r)}{r}, $$
where $\beta(r)=1$ in $d\geq 3$ and $\beta(r)=\log(1+r)$ in $d=2$. We remark that this estimate is a quantitative justification of the linearization ansatz of Caracciolo et al.\ in \cite{CaLuPaSi14} from a macroscopic scale down to a microscopic scale (see Proposition \ref{prop:quenched} for a quenched version).

Finally, we can combine these two steps to prove Theorem \ref{thm:intro1} and  Theorem\ref{thm:intro2}. Additionally, we prove a variant of these two theorems where we can relax the condition on $\gamma$ to $\gamma>\frac{d}{2}-1$ which is the natural condition for $\nabla \Psi$, cf.\ Remark \ref{rem:sobwhiteDir}. However, in order to achieve this we have to ignore or average out the microscopic scales at which $\mu^{R,L}$ has no better regularity than Dirac measures. 
We refer to Theorem \ref{thm:maincvg} for the precise statement.

\begin{remark}\label{rem:macroscale}
 Another way to deal with the global non-existence of $\nabla\Psi$ in $d=2$  would be to keep $L$ fixed and let only $R$ tend to $\infty$. Upon changing notation and adapting the definition of $\nabla\Psi$ to live on the torus $Q_L$ our estimates in Sections \ref{sec:linear} and \ref{sec:conv} allow to deduce such a result for $d=2,3$. We leave the details to the reader.

 However, in dimension one for $L$ being fixed,  based on the explicit form of the optimal coupling one can directly argue the convergence of the properly rescaled displacements to a Brownian bridge, see e.g.\ \cite{CaSi14, dBGiUt05}. The key observation is that the solution to the optimal matching problem on $[0,1)$ with $n$ iid uniform points $X_1,\ldots, X_n$ maps the interval $[\frac{i-1}{n},\frac{i}{n})$ to the i-th point in the \emph{ordered} tuple $\tilde X_1\leq \tilde X_2\leq\ldots\leq\tilde X_n$.  The random variable $\tilde X_i$ follows a  $\mathsf{Beta}(i,n-i+1)$ distribution from which one can derive convergence of the rescaled displacements $\sqrt{n}\bra{\tilde X_i-\frac{i}{n}}$ to a Brownian bridge.
\end{remark}

\begin{remark}
 Essentially identical results to Theorem \ref{thm:intro1} and Theorem \ref{thm:intro2} hold if we replace $\mu^{R,L}$ by a Poisson point process of intensity $R^d$, i.e.\ we replace the deterministic number of iid points  by a random number $N$ of iid uniformly distributed points in $Q_L$ where $N$ is Poisson distributed with parameter $(RL)^d$. On a technical side, this requires to add several additional estimates to account for the fluctuation of $N$ around its mean $(RL)^d$ which can be dealt with using Chernoff bounds similar to, e.g.,  \cite{GolTre}.
\end{remark}

\section{Preliminaries}

Throughout the paper we will use the following notation.
We write $Q_L=\lt[-\frac{L}{2},\frac{L}{2}\rt)^d$ and denote the ball or radius $\ell$ by $B_\ell$. The indicator function of a set $A$ will be denoted by $\chi_A$.
We fix a radially symmetric $\eta \in C^\infty_c(B_1)$ and for $\eps>0$ we write $\eta_\eps(x)=\frac{1}{\eps^d}\eta\lt(\frac{x}{\eps}\rt).$ For convolution with $\eta_\eps$ we will use the shorthand notation $\eta_\eps * u=u_\eps$.
For a random variable $U$ we write $\mathsf{Law}(U)$ for its distribution. If $\Sigma$ is a symmetric and positive definite matrix we write $\mathcal{N}(0,\Sigma)$ for the centered Gaussian distribution with covariance matrix $\Sigma$. For two measure $\mu,\nu$ on $\R^d$ of the same finite mass we denote the set of all couplings between $\mu$ and $\nu$ by $\mathsf{Cpl}(\mu,\nu)$ and the $L^p$ Wasserstein distance by
$$ W_p(\mu,\nu)=\inf_{\pi\in\mathsf{Cpl}(\mu,\nu)}\lt(\int |x-y|^p d\pi\rt)^\frac1p.$$
The $L^2$ Wasserstein distance on the torus will be denoted by $W_{per}$.
For a set $B\subset \R^d$ we write 
\begin{equation}\label{notationWkappa}
W_B(\mu,\kappa)=W_2\lt(\mu\restr B,\frac{\mu(B)}{\mathsf{Leb}(B)}\mathsf{Leb}\rt),
\end{equation}
where $\mathsf{Leb}$ denotes the Lebesgue measure. We usually write $dx$ for integration with respect to $\mathsf{Leb}.$

The notation $A\ll 1$, which we only use in assumptions, means that there exists an $\eps>0$ only depending on the dimension, the fixed cutoff function $\eta$ and the parameters $p$ and $\gamma$, such that
if $A\leq \eps$ then the conclusion holds.
 Similarly, the notation $A\lesssim B$, which we use in output statements, means that there exists a global
constant $C>0$ depending on the dimension, the fixed cutoff function $\eta$ and the parameters $p$ and $\gamma$ such that $A\le C B$. We write $A\sim B$ if $A\les B\les A$.

\subsection{Curl-free GFF}\label{sec:GFGF}
 For $L,R>0$, we let $\mu^{R,L}$ be a normalized Binomial point process of intensity $R^d$, i.e.
\begin{equation}\label{def:muRL}
 \mu^{R,L}=\frac{1}{R^d}\sum_{i=1}^{(RL)^d} \delta_{X_i},
\end{equation}
where $X_i$ are iid random variables uniformly distributed on $Q_L$. In particular, here and in the rest of the article we always assume that $(RL)^d\in \N$. We often identify $\mu^{R,L}$ and its $Q_L$ periodic extension to $\R^d$.
We  define the $Q_L-$periodic measure 
\begin{equation}\label{def:WRL}
 W^{R,L}=R^{\frac{d}{2}}\lt( \mu^{R,L}-  1\rt).
\end{equation}
We consider $u^{R,L}$ the $Q_L-$periodic solution with average zero of
\begin{equation}\label{def:uRL}
 \Delta u^{R,L}=W^{R,L}.
\end{equation}
Notice that formally we have the rescaling\footnote{Here and in the rest of the article we use the obvious coupling of the stochastic quantities. We will do so without explicitly mentioning this anymore.}
\begin{equation}\label{rescalingnabu}
  W^{R,L}(x)= R^{-\frac{d}{2}} W^{1, LR}(Rx) \qquad \textrm{and} \qquad \nabla u^{R, L}(x)=R^{\frac{d}{2}-1} \nabla u^{1,L R}(Rx).
\end{equation}
Moreover, due to periodicity quantities like $\mu^{R,L}_\eps(x), \nabla u^{R,L}_\eps(x), W^{R,L}_\eps(x)$ will all be stationary.
 If  $W$ is the white noise on $\R^d$ we formally consider the solution $\Psi$ to
\[
 \Delta \Psi= W \qquad \textrm{ in } \R^d.
\]
The curl-free Gaussian free field (or curl-free GFF) is then, still formally, the random distribution $\nabla \Psi$. More precisely if $d\ge 3$, $\nabla \Psi$
is a random distribution such that for every $f\in C^{\infty}_c(\R^d,\R^d)$,
\begin{equation}\label{defnablapsi}
 \mathsf{Law}(\nabla\Psi(f))= \mathcal{N}\lt(0, \int (\phi^{\infty})^2\rt)
\end{equation}
where $\phi^{\infty}$ is the unique $L^2(\R^d)$ solution to
\begin{equation}\label{defphi}
 -\Delta \phi^\infty=\nabla\cdot f \qquad \textrm{ in } \R^d.
\end{equation}
When $d=2$, in general solutions of \eqref{defphi} are not in $L^2$ unless $\int f=0$ and thus some renormalization is needed. 
Let $\eta$ be our usual smooth cut-off function with $\int\eta=1$. We define $\nabla \Psi-\nabla \Psi_1(0)$ as follows: for every $f\in C^\infty_c(\R^d;\R^d)$ let $\phi^\infty$ be the unique $L^2(\R^2)$ solution of 
\begin{equation}\label{defphi2d}
 -\Delta \phi^\infty =\nabla\cdot \lt(f -\eta \int f\rt) \qquad \textrm{in } \R^d,
\end{equation}
which exists since $\int (f -\eta \int f)=0$. We then require 
\[
 \mathsf{Law}\left((\nabla \Psi-\nabla \Psi_1(0))(f)\right)=\mathcal{N}\lt(0,\int (\phi^\infty)^2\rt).
\]
Notice that at least formally, this notation is consistent with our convention $\nabla \Psi_1(0)=\int \eta \nabla \Psi$.

\begin{remark}
 Let us point out that the existence of the curl-free GFF may be seen as a consequence of our convergence result Theorem \ref{theo:convlinearbody} below. 
\end{remark}

\begin{remark}\label{rem:GFF}
  If $W_{i,j}$ are independent white noises and $W_i=(W_{i,1},\cdots, W_{i,d})$, we can define $h=(h_1,\cdots, h_d)$ the vector GFF  as the solution of (see for instance \cite{AKMbook,Nathanael})
  \[
 \Delta  h_i= \nabla \cdot W_i \qquad \forall i=1,\cdots d.
  \]
  It is then not hard to see that $\Delta \Psi=\nabla \cdot h$ and thus in terms of Hodge decompositions, $\nabla \Psi$ corresponds to the curl-free part of $h$ (see also \cite[Th. 1.2.5]{Aru}). 
\end{remark}

We close this section with two estimates which will be used in Section \ref{sec:conv}.

\begin{lemma}\label{lem:boundmu}
 For every $p\ge 2$, $L\gg r>0$ and $x\in Q_L$,
 \begin{equation}\label{eq:momentmu}
 \EE\lt[ |W_r^{1,L}(x)|^p\rt]^{\frac{1}{p}}\les 
 \frac{1}{r^{\frac{d}{2}}}\lt(1+\frac{1}{r^{\frac{d(p-2)}{2p}}}\rt).  
\end{equation}

\end{lemma}
\begin{proof}
By stationarity, it is enough to prove \eqref{eq:momentmu} for $x=0$. If  $X_i$ are iid random variables uniformly distributed in $Q_L$,
\[
 W_r^{1,L}(0)=\sum_{i=1}^{L^d} \lt( \eta_r(X_i) -\frac{1}{|Q_L|}\rt).
\]
Letting $Y_i= \eta_r(X_i) -\frac{1}{|Q_L|}$ and using Rosenthal's inequality \cite[Th. 3]{Rosenthal} for sums of centered iid random variables (recall $\int \eta=1$), we obtain
\[\EE\lt[ |W_r^{1,L}(x)|^p\rt]^{\frac{1}{p}}\les  \lt(L^d\EE\lt[\lt|Y_i\rt|^2\rt]\rt)^{\frac{1}{2}}+ \lt( L^d \EE\lt[\lt|Y_i\rt|^p\rt]\rt)^{\frac{1}{p}}
\]
Now by definition, we have for every $p\ge 1$
\[
 L^d\EE\lt[\lt|Y_i\rt|^p\rt]=\int_{Q_L} \lt|\eta_r -\frac{1}{|Q_L|}\rt|^p\les \frac{1}{r^{d(p-1)}} +\frac{1}{L^{d(p-1)}}\les \frac{1}{r^{d(p-1)}},
\]
which concludes the proof.
\end{proof}

\subsection{Negative Sobolev spaces}\label{sec:negSob}
Fix $p\ge 1$ and  $\gamma>0$ with $\gamma= k+s$, $k\in \N$ and $s\in(0,1)$. Let us stress the fact that we only consider here the case $\gamma\notin \N$ since this will be enough for our purpose thanks to Sobolev embedding, see Remark \ref{rem:Sob}. The case $\gamma\in \N$ is special and would require to be treated a bit differently, see \cite[Prop. D.1]{AKMbook}. For $u:\R^d\to \R^n$, we define the homogeneous $W^{\gamma,p}$ semi-norm of $u$ by
\[
[u]_{W^{\gamma,p}}^p=\int_{\R^d\times \R^d} \frac{|\nabla^k u(x)-\nabla^k u(y)|^p}{|x-y|^{d+sp}} 
\]
  and then the $W^{\gamma,p}$ norm  as 
\[
 \|u\|_{W^{\gamma,p}}^p=\sum_{i=0}^k \int |\nabla^i u|^p + [u]_{W^{\gamma,p}}^p.
\]
For every $\ell>0$, we define by duality (here $p'$ is the H\"older conjugate of $p$ i.e. $\frac{1}{p}+\frac{1}{p'}=1$),
\begin{equation}\label{def:negSob}
 \|u\|_{W^{-\gamma,p}(B_\ell)}=\sup\lt\{ \int u v \ : v\in C^{\infty}_c(B_\ell), \, \|v\|_{W^{\gamma,p'}}\le 1 \rt\}.
\end{equation}
We say that a sequence of random distributions $(u_n)_{n\geq 1}$ converges in law in $W^{-\gamma,p}_{\loc}$ to a limit distribution $u$ if for every $\ell>0$ and every $F\in C^0_b(W^{-\gamma,p}(B_\ell),\R)$
\[
 \lim_{n\to \infty} \EE[F(u_n)]=\EE[F(u)].
\]
For fixed $p, \gamma$ and $\ell$, we say that the sequence $(u_n)_{n\geq 1}$ has bounded moments of arbitrary order in $W^{-\gamma,p}(B_\ell)$ if for every $q\ge 1$,
\[
 \sup_n \EE[\|u_n\|^q_{W^{-\gamma,p}(B_\ell)}]<\infty.
\]
\begin{remark}\label{rem:Sob}
 Notice that by Sobolev embedding we have for $\gamma'<\gamma$ and $q\ge p$
 \[
  \|u\|_{W^{-\gamma,p}(B_\ell)}\le C  \|u\|_{W^{-\gamma',q}(B_\ell)}
 \]
for a constant $C>0$ depending on $\gamma, \gamma', p, q, \ell$. Therefore, in order to control arbitrary large $q-$moments in $W^{-\gamma,p}(B_\ell)$ norms,
it is enough to bound $\EE[\|u\|_{W^{-\gamma,p}(B_\ell)}^p]$ for every $p\ge 1$.
\end{remark}
We now give a characterization of negative Sobolev norms in terms of averages. Similar characterizations are widely used in the literature (see the book \cite{Triebel} and  an example of applications in SPDEs \cite{OtWe}) but we could not find exactly what we needed. In particular, classically the regularization is through a semi-group or a kernel having enough cancellations (see \eqref{condrho} below). Here we use instead a standard (positive) convolution kernel. This is due to the fact that we are only interested in negative Sobolev norms and not in a unified characterization for both positive and negative regularity exponents. Another difficulty is that we need a characterization of local Sobolev spaces where the situation is more subtle than in the whole space case.  
We follow the general strategy of \cite[Prop. D5]{AKMbook}. Our starting point is the following variant of \cite[Prop. D3]{AKMbook}:
\begin{lemma}\label{lem1.1}
 Let $\gamma>0$ with $\gamma=k+s$ with $k\in \N$ and $s\in(0,1)$. Then, if $\rho\in C^\infty_c(\R^d)$ is rotationally invariant with 
 \begin{equation}\label{condrho}
  \int x_i^l \rho(x) dx=0 \qquad \forall i\in[1,d] \ \textrm{ and } \ l\in[0,k],
 \end{equation}
then\footnote{ Here with a slight abuse of notation, $\rho_\eps=\eps^{-d}\rho(\cdot/\eps)$} 
\begin{equation}\label{primal}
 \int_0^\infty \eps^{-\gamma p} \int_{\R^d} |u\ast \rho_\eps|^p \frac{d\eps}{\eps}\les [u]_{W^{\gamma,p}}^p.
\end{equation}

\end{lemma}
Note that necessarily the function $\rho$ in Lemma \ref{lem1.1} cannot have a  constant sign due to condition \eqref{condrho}.

\begin{proof}
 For a $j$ tensor $A$ and $\xi\in \R^d$, we denote $A(\xi^{\otimes j})= A(\xi,\cdots,\xi)$. By Taylor formula, for every $x\in \R^d$, 
  \begin{align*}
&  u\ast \rho_\eps(x)= \int_{\R^d} \rho_\eps(x-y)u(y)dy\\
  &=\int_{\R^d} \rho_\eps(x-y)\lt[u(x) + \sum_{j=0}^{k-1} \frac{1}{j!} \nabla^j u(x)((y-x)^{\otimes j})\rt.\\
  &\qquad \qquad\lt. +\int_0^1 \frac{1}{k!} \nabla^{k} u(x+t(y-x))((y-x)^{\otimes k})(1-t)^{k-1}dt\rt] dy\\
  &\stackrel{\eqref{condrho}}{=}\frac{1}{k!}\int_{\R^d} \rho_\eps(x-y)\int_0^1 [\nabla^{k} u(x +t(y-x))-\nabla^{k} u(x)]((y-x)^{\otimes k})(1-t)^{k-1} dt dy. 
  \end{align*}
Therefore, using Jensen's inequality we find 
\begin{align*}
 \int_{\R^d}|u\ast \rho_\eps|^p&\les \int_0^1 \int_{\R^d\times\R^d} |x-y|^{pk} \rho_\eps(x-y) |\nabla^{k} u(x +t(y-x))-\nabla^{k} u(x)|^p(1-t)^{(k-1)p}.
\end{align*}
Since $\rho$ is rotationally invariant writing $\bar\rho_1(t)=\rho(t\frac{x}{|x|})$ we have
\[
 \int_0^\infty \eps^{-p\gamma}\rho_\eps(x-y) \frac{d\eps}{\eps}= \frac{1}{|x-y|^{d+ p\gamma}} \int_0^\infty  t^{-d-p \gamma} \bar\rho_1(1/t) \frac{dt}{t}\les  \frac{1}{|x-y|^{d+ p\gamma}}.
\]
Thus, recalling that $\gamma=k+s$,
\begin{align*}
 \int_0^\infty \eps^{-p\gamma} \int_{\R^d} |u\ast \rho_\eps|^p \frac{d\eps}{\eps}
 &\les \int_0^1 (1-t)^{(k-1)p}\int_{\R^d\times\R^d}  \frac{|\nabla^{k} u(x +t(y-x))-\nabla^{k} u(x)|^p}{|x-y|^{d+ p(\gamma -k)}} \\
 &=\int_0^1 (1-t)^{(k-1)p}\int_{\R^d\times\R^d}  \frac{|\nabla^{k} u(x +t(y-x))-\nabla^{k} u(x)|^p}{|x-y|^{d+ ps}} \\
 &=\int_0^1 t^{d+p s}(1-t)^{(k-1)p}\int_{\R^d\times\R^d}  \frac{|\nabla^{k} u(x +t(y-x))-\nabla^{k} u(x)|^p}{|t(y-x)|^{d+ ps }} \\
 &\stackrel{z= t(y-x)}{=} \int_0^1 t^{p s}(1-t)^{(k-1)p}\int_{\R^d\times\R^d}  \frac{|\nabla^{k} u(x +z)-\nabla^{k} u(x)|^p}{|z|^{d+ ps}} \\
 &\les [u]_{W^{\gamma,p}}^p.
\end{align*}

\end{proof}
\begin{remark}
 In the case $p=2$, \eqref{primal} may be also obtained via Fourier transform.
\end{remark}

We may now prove the main result of this section. We recall that we fixed a radial cut-off function $\eta$ and that we use the notation $u_\eps=u\ast \eta_\eps$.
\begin{theorem}\label{theo:mainnegsob}
 For every $\gamma>0$ with $\gamma\notin \N$ and $p\ge 1$, if $\ell\gg1$ then  
 \begin{equation}\label{eq:main}
  \|u\|_{W^{-\gamma,p}(B_\ell)}^p\les  \int_0^1 \eps^{p\gamma} \int_{B_{2\ell}} |u_\eps|^p  \frac{d\eps}{\eps}.
 \end{equation}

\end{theorem}
\begin{proof}
 Let  $k=\lceil-\gamma\rceil$. Choose then $\bar \eps\in (1/2,1)$ such that (which exists by the mean value theorem) 
 \begin{equation}\label{goodT}
  \int_{B_{2\ell}} | u_{\bar \eps}|^p \les \int_0^1 \eps^{p\gamma} \int_{B_{2\ell}} | u_\eps|^p  \frac{d\eps}{\eps}.
 \end{equation}
For a fixed function $f\in C^\infty_c(B_\ell)$ we let 
\[
 F(\eps)=\int_{\R^d} u (\eta_\eps \ast_{k+1} \eta_\eps ) \ast f
\]
where $\eta_\eps \ast_{k+1} \eta_\eps =\eta_\eps\ast\cdots \ast \eta_\eps$ ($\eta_\eps$ convolved $k+1$ times with itself). Our aim is to estimate $F(0)$. For this we use Taylor expansion and write 
\begin{equation}\label{Psi0}
 F(0)=\sum_{j=0}^{k-1} \frac{(-\bar \eps)^j}{j!} \frac{d^j}{d\eps^j}F(\bar \eps) +\frac{(-1)^k}{(k-1)!}\int_0^{\bar \eps}  \eps^{k-1} \frac{d^k}{d\eps^k}F(\eps) d\eps.
\end{equation}
We claim that for every $j\in [0,k]$,
\[
 \eps^j \frac{\partial^j}{\partial \eps^j} (\eta_\eps \ast_{k+1} \eta_\eps )= \eta_\eps\ast \rho^j_\eps.
\]
For some kernel  $\rho_\eps^j(x)=\frac{1}{\eps^d}\rho_1^j(x/\eps)$ with $\rho_1^j$ satisfying  \eqref{condrho}  for $l \in [0,j]$ and such that $\spt \rho_1^j\subset B_k$.
Indeed, 
\[
 \eps^j \frac{\partial^j}{\partial \eps^j} (\eta_\eps \ast_{k+1} \eta_\eps )=\sum_{\sum_{i=1}^{k+1} \alpha_i=j} (\eps^{\alpha_1} \frac{\partial^{\alpha_1}}{\partial \eps^{\alpha_1}} \eta_\eps)
 \ast\cdots \ast (\eps^{\alpha_{k+1}} \frac{\partial^{\alpha_{k+1}}}{\partial \eps^{\alpha_{k+1}}} \eta_\eps).
\]
Since $\alpha_i$ are integers and $j<k+1$, for every $(\alpha_1,\cdots, \alpha_{k+1})$ there is always at least some $i$ such that $\alpha_i=0$. Therefore, if we can prove that for every $\alpha$, 
\[
 \eps^{\alpha} \frac{\partial^{\alpha}}{\partial \eps^{\alpha}} \eta_\eps=\tilde\rho_\eps^\alpha
\]
for some $\tilde\rho_1^\alpha$ satisfying \eqref{condrho} for $l \in [0,\alpha]$ and $\spt \tilde\rho_1^\alpha\subset B_1$ then the claim would follow either looking in Fourier space or using that $u_\eps\ast v_\eps= ((u\ast v)_\eps)_\eps$. The fact that $\spt \tilde\rho_\eps^\alpha\subset B_\eps $ is immediate since $\spt \eta_\eps\subset B_\eps$. Looking in Fourier space we have 
\[
 \widehat{\lt(\eps^{\alpha} \frac{\partial^{\alpha}}{\partial \eps^{\alpha}} \eta_\eps\rt)}(\xi)= \nabla^\alpha \hat{\eta}(\eps \xi)(\eps\xi, \cdots, \eps \xi)
\]
and the claim follows letting $\rho_1^\alpha=\lt(\nabla^\alpha \hat{\eta}(\xi)(\xi, \cdots, \xi) \rt)^{\vee}$ recalling that $\int f=\hat f(0)$. \\
We may now estimate the various terms in \eqref{Psi0}. For every $j\le k$, we have 
\begin{align*}
  \eps^j \frac{d^j}{d\eps^j}F(\eps)&=\int_{\R^d} u  \lt( \eps^j \frac{\partial^j}{\partial \eps^j} (\eta_\eps \ast_{k+1} \eta_\eps ) \ast f\rt) \\
  &=\int_{B_{2\ell}} u_\eps (f\ast \rho_\eps^j).
\end{align*}
In particular, for $j<k$,
\begin{multline}\label{firstpiece}
 \lt|\frac{d^j}{d\eps^j}F(\bar \eps)\rt|\les \| u_{\bar \eps}\|_{L^p(B_{2\ell})} \|f\ast \rho_{\bar \eps}^j\|_{L^{p'}(B_{2\ell})}\\
 \les \| u_{\bar \eps}\|_{L^p(B_{2\ell})} \|f\|_{L^{p'}(B_{2\ell})}\stackrel{\eqref{goodT}}{\les}\lt(\int_0^1 \eps^{p\gamma} \int_{B_{2\ell}} | u_\eps|^p  \frac{d\eps}{\eps}\rt)^{\frac{1}{p}}\|f\|_{L^{p'}}.
\end{multline}
We now estimate the last term as
\begin{align*}
\int_0^{\bar \eps}  \eps^{k-1} \frac{d^k}{d\eps^k}F(\eps) d\eps&= \int_0^{\bar \eps} \int_{B_{2\ell}} u (\eta_\eps\ast \rho_\eps^k \ast f) \frac{d\eps}{\eps}\\
&=\int_0^{\bar \eps} \int_{B_{2\ell}} u_\eps  \rho_\eps^k \ast f \frac{d\eps}{\eps}\\
&\le \int_0^{\bar \eps} \lt(\int_{B_{2\ell}} |u_\eps|^p\rt)^{\frac{1}{p}}\lt(\int_{B_{2\ell}} |f\ast \rho_\eps^k|^{p'}\rt)^{\frac{1}{p'}} \frac{d\eps}{\eps}\\
&\le \lt(\int_0^{\bar \eps}  \eps^{p\gamma} \int_{B_{2\ell}} |u_\eps|^p \frac{d\eps}{\eps}\rt)^{\frac{1}{p}} \lt(\int_0^{\bar \eps}  \eps^{-p'\gamma} \int_{B_{2\ell}} |f\ast \rho_\eps^k|^{p'} \frac{d\eps}{\eps}\rt)^{\frac{1}{p'}}
\\
&\stackrel{\eqref{primal}}{\les} \lt(\int_0^{1}  \eps^{p\gamma} \int_{B_{2\ell}} |u_\eps|^p \frac{d\eps}{\eps}\rt)^{\frac{1}{p}} [f]_{W^{\gamma,p'}}.
\end{align*}
Putting this together with \eqref{firstpiece} yields \eqref{eq:main}.
\end{proof}
\begin{remark}
 In the whole space case, \eqref{eq:main} may be either obtained  by a simple computation using Fourier transform when  $p=2$, 
 or a more involved one using Littlewood-Paley decompositions for general $p$ (see \cite{Triebel}). 
 However,  it is not clear to us how to obtain the  local version \eqref{eq:main} using computations in the spectral domain. 
\end{remark}
As a first illustration of the use of \eqref{eq:main} we may show:
\begin{remark}\label{rem:sobwhiteDir}
 On the one hand, a Dirac mass is in $W^{-\gamma,p}$ 
 if\footnote{It is actually also a necessary condition as one can directly see from definition \eqref{def:negSob}.} $\gamma>d(1-\frac{1}{p})$. On the other hand, 
 the white noise $W$ on $\R^d$ satisfies $\EE[\|W\|_{W^{-\gamma,p}(B_\ell)}^p]<\infty$ for every $\gamma> \frac{d}{2}$ and every $1\le p<\infty$.
 In particular,  convergence in law in $W^{-\gamma,p}$ of  $W^{R,L}$ (recall  \eqref{def:WRL})
 to white noise cannot hold under the mere assumption $\gamma> \frac{d}{2}$. The two possible solutions are  either to impose stronger conditions on $(p,\gamma)$ or to ignore 
 the microscopic scales at which $W^{R,L}$ has no better regularity than Dirac masses. The same observation applies when considering the convergence
 of $\nabla u^{R,L}$ (recall  \eqref{def:uRL}) to the curl-free GFF 
\end{remark}
\begin{proof}
 The computation for the Dirac mass is immediate and we leave it to the reader. As for the white noise $W$, by \eqref{eq:main} and  stationarity we have 
 \[
 \EE[\|W\|_{W^{-\gamma,p}(B_\ell)}^p]\les \int_0^1 \eps^{p\gamma}  \int_{B_{2\ell}} \EE[|W_\eps|^p] \frac{d\eps}{\eps}= |B_{2\ell}| \int_0^1 \eps^{p\gamma}\EE[|W_\eps(0)|^p] \frac{d\eps}{\eps}.
 \]
Now since $W_\eps(0)=W(\eta_\eps)=\mathcal{N}(0, \|\eta_\eps\|_{L^2}^2)$ and since Gaussian random variables satisfy a reverse H\"older inequality, we have 
\[
 \EE[|W_\eps(0)|^p]\les \EE[|W_\eps(0)|^2]^{\frac{p}{2}}\les \|\eta_\eps\|_{L^2}^p\les \frac{1}{\eps^{p\frac{d}{2}}}.
\]
Finally,
\[\EE[\|W\|_{W^{-\gamma,p}(B_\ell)}^p]\les |B_{2\ell}| \int_0^1 \eps^{p(\gamma-\frac{d}{2})} \frac{d\eps}{\eps},\]
which is indeed finite if $\gamma>\frac{d}{2}$.
\end{proof}
The previous remark motivates the following simple but useful lemma.
\begin{lemma}\label{lem:removesmallscale}
 For every $p\ge 1$, $\gamma>0$ with $\gamma\notin \N$, $\ell\gg 1$ and  $t\in (0,1)$,
\begin{equation}\label{eq:removesmallscale}
 \|u_t\|_{W^{-\gamma,p}(B_\ell)}^p\les  t^{p\gamma} \int_{B_{3\ell}} |u_t|^p + \int_t^1 \eps^{p\gamma} \int_{B_{3\ell}}  |u_\eps|^p \frac{d\eps}{\eps}.
\end{equation}
 
\end{lemma}
\begin{proof}
 By \eqref{eq:main},
 \[
 \|u_t\|_{W^{-\gamma,p}(B_\ell)}^p\les  \int_0^1 \eps^{p\gamma} \int_{B_{2\ell}} |u\ast\eta_t\ast\eta_\eps|^p  \frac{d\eps}{\eps}.
 \]
Observe that for every $v$ and every $\rho$ with $\spt \rho\subset B_1$ and $\int \rho=1$
\[
 \int_{B_{2\ell}} |v\ast\rho|^p\le \int_{\R^d} |(v\chi_{B_{3 \ell}})\ast \rho|^p\stackrel{\textrm{Young}}{\les} \int_{B_{3 \ell}} |v|^p.
\]
Using this observation with $v=u\ast\eta_t$, $\rho=\eta_\eps$ if $\eps\in(0,t)$ and $v=u\ast \eta_\eps$, $\rho=\eta_t$ if $\eps\in (t,1)$ concludes the proof of \eqref{eq:removesmallscale}.
\end{proof}

\section{Linear theory}\label{sec:linear}
The aim of this section is to prove the distributional convergence of $\nabla u^{R,L}$ (recall the definition \eqref{def:uRL}) to $\nabla \Psi$, the curl-free GFF. 
When $d=2$, this convergence holds only after subtracting the  logarithmically diverging  (random) constant  $\nabla u^{R,L}_1(0)$.
 The proof is  divided in three steps. The first is to obtain moment bounds for the single observables $\nabla u_\eps(x)$. 
With the help of Theorem \ref{theo:mainnegsob} and Lemma \ref{lem:removesmallscale}, these imply tightness in the appropriate negative Sobolev spaces which by Prokhorov yield convergence in law to some limit distribution. 
The last step is to identify this limit distribution as $\mathsf{Law}(\nabla \Psi)$.

\subsection{Moment bounds}\label{sec:momentbounds}

Recall the definition \eqref{def:uRL} of $u^{R,L}$. 

\begin{lemma}\label{lem:boundnablu}
 For every $p\ge 2$,  $L\gg r>0$ and $x\in Q_L$,
 \begin{equation}\label{eq:momentnabu}
  \EE\lt[ |\nabla u_r^{1,L}(x)|^p\rt]^{\frac{1}{p}}\les\lt( 1+ \frac{1}{r^{\frac{d(p-2)}{2p}}}\rt)\begin{cases}
                                                           \frac{1}{r^{\frac{1}{2}(d-2)}} & \textrm{ if } d\ge 3\\[8pt]
                                                           \log^{\frac{1}{2}}\lt(\frac{L}{r}\rt) &\textrm{ if } d=2.
                                                           \end{cases}
 \end{equation}
Moreover, if $d=2$, for every $p\ge 2$, $L\gg R \ge r>0$,  and $x,y \in Q_L$,
\begin{equation}\label{eq:momentnabulog}
\EE\lt[ |\nabla u_r^{1,L}(x)- \nabla u_R^{1,L}(y)|^p\rt]^{\frac{1}{p}}\les\lt( 1+ \frac{1}{r^{\frac{p-2}{p}}}\rt)\lt( \log^{\frac{1}{2}}\lt(\frac{R+|x-y|}{r}\rt)+1\rt).
\end{equation}

\end{lemma}

\begin{proof}
 We first recall some properties of the Green function of the Laplacian in the periodic domain $Q_L$. Let $G^L$ be the $Q_L-$periodic solution of average zero of 
 \[
  \Delta G^L=\delta_0- \frac{1}{|Q_L|}.
 \]
By scaling we have $G^L(x)= \frac{1}{L^{d-2}}G^1\lt(\frac{x}{L}\rt)$. We let $G_{\R^d}$ be the Green function of the Laplacian on $\R^d$ and fix $\chi$
a smooth cut-off function with compact support in $Q_1$ and such that $\chi=1$ in a neighborhood of zero. We identify $\chi G_{\R^d}$ with its $Q_1-$periodic extension. Using that $\Delta (G^1-\chi G_{\R^d})$ is a smooth periodic function,
we find that we may write,
\[
 \nabla G^1 =\chi\nabla G_{\R^d} + h
\]
 in $Q_L$ for some smooth function $h$. This yields 
 \begin{equation}\label{eq:decompGreen}
  \nabla G^L=\chi(\cdot/L)\nabla G_{\R^d} +\frac{1}{L^{d-1}}h(\cdot/L).
 \end{equation}

 We now start the proof of \eqref{eq:momentnabu}. 
If  $(X_i)_{i}$ are iid random variables uniformly distributed in $Q_L$, we have $\nabla u^{1,L}(x)=\sum_{i=1}^{L^d} \nabla G^L(X_i-x)$. If we introduce for each fixed $r$ and $x$, the iid random variables
\begin{equation}\label{eq:defY}
 Y_{i,r}(x)=\int_{Q_L} \eta_r(y-x)\nabla G^L(X_i-y) dy,
\end{equation}
we thus have  $\nabla u_r^{1,L}(x)=\sum_{i=1}^{L^d} Y_{i,r}(x)$.
Using Rosenthal's inequality as in the proof of \eqref{eq:momentmu}, we get
\[
 \EE\lt[ |\nabla u_r^{1,L}(x)|^p\rt]^{\frac{1}{p}}\les \lt(L^d \EE[|Y_{i,r}(x)|^2]\rt)^{\frac{1}{2}} + \lt( L^d \EE[|Y_{i,r}(x)|^p]\rt)^{\frac{1}{p}}.
\]
The proof of \eqref{eq:momentnabu} is concluded provided we can show that for every $p,d\ge 2$  and every $x\in Q_L$,
\begin{equation}\label{eq:boundpnab}
 \lt( L^d \EE[|Y_{i,r}(x)|^p]\rt)^{\frac{1}{p}}\les \begin{cases}\frac{1}{r^{d-1-\frac{d}{p}}} &\textrm{ if } (p,d)\neq(2,2)\\
    \log^{\frac{1}{2}}\lt(\frac{L}{r}\rt) &\textrm{ if } (p,d)=(2,2).                                                                                                                                                                                                                                                                                     \end{cases}
\end{equation}
By stationarity, we may  assume that $x=0$.  Since $X_i$ is uniformly distributed, we have 
\[
 \lt( L^d \EE[|Y_{i,r}(0)|^p]\rt)^{\frac{1}{p}}=\lt( \int_{Q_L} \lt|\int_{Q_L} \eta_r(y)\nabla G^L(X-y) dy\rt|^p dX\rt)^{\frac{1}{p}}.
\]
Recalling the decomposition \eqref{eq:decompGreen} of $\nabla G^L$ and observing that since $| h|\les 1$,
\begin{equation}\label{eq:restGreen}
 \lt( \int_{Q_L} \lt|\int_{Q_L} \frac{1}{L^{d-1}}\eta_r(y) h\lt(\frac{X-y}{L}\rt) dy\rt|^p dX\rt)^{\frac{1}{p}}\les  \frac{1}{L^{d-1-\frac{d}{p}}}\le  \frac{1}{r^{d-1-\frac{d}{p}}},
\end{equation}
it is enough to prove \eqref{eq:boundpnab} with $\chi(\cdot/L)\nabla G_{\R^d}$ instead of $\nabla G^L$.
Using that $0\le \chi\le 1$ and $\nabla G_{\R^d}(x)=C \frac{x}{|x|^{d}}$, we have 
\begin{multline*}
 \lt( \int_{Q_L} \lt|\int_{Q_L} \eta_r(y)\chi\lt(\frac{X-y}{L}\rt)\nabla G_{\R^d}(X-y) dy\rt|^p dX\rt)^{\frac{1}{p}}\\
 \les \lt( \int_{B_{2r}} \lt(\int_{Q_L} \eta_r(y)\frac{1}{|X-y|^{d-1}} dy\rt)^p dX\rt)^{\frac{1}{p}}
 + \lt( \int_{Q_L\backslash B_{2r}} \lt(\int_{Q_L} \eta_r(y)\frac{1}{|X-y|^{d-1}} dy\rt)^p dX\rt)^{\frac{1}{p}}.
\end{multline*}
For the first term we use that for every\footnote{Notice that if $\eta$ is assumed to be radially decreasing, then the left-hand side is actually maximized at $X=0$.} $X\in \R^d$,
\begin{equation}\label{eq:Ybounded}
 \int_{Q_L} \eta_r(y)\frac{1}{|X-y|^{d-1}} dy\les \int_{Q_L} \eta_r(y)\frac{1}{|y|^{d-1}} dy\les \frac{1}{r^{d-1}}
\end{equation}
to obtain
\begin{equation}\label{eq:estimclose}
 \lt( \int_{B_{2r}} \lt(\int_{Q_L} \eta_r(y)\frac{1}{|X-y|^{d-1}} dy\rt)^p dX\rt)^{\frac{1}{p}}\les \frac{1}{r^{d-1-\frac{d}{p}}}.
\end{equation}
For the second term we use that if $y\in B_r$ and $X\in B_{2r}^c$, then 
\begin{equation*}\label{eq:monopole}
 \frac{1}{|X-y|^{d-1}}\les \frac{1}{|X|^{d-1}}
\end{equation*}
to obtain similarly
\begin{multline}\label{eq:estimfar}
 \lt( \int_{Q_L\backslash B_{2r}} \lt(\int_{Q_L} \eta_r(y)\frac{1}{|X-y|^{d-1}} dy\rt)^p dX\rt)^{\frac{1}{p}}\les \lt( \int_{Q_L\backslash B_{2r}}\frac{1}{|X|^{p(d-1)}}  dX\rt)^{\frac{1}{p}}\\
 \les \begin{cases}\frac{1}{r^{d-1-\frac{d}{p}}} &\textrm{ if } (p,d)\neq(2,2)\\[8pt]
    \log^{\frac{1}{2}}\lt(\frac{L}{r}\rt) &\textrm{ if } (p,d)=(2,2). 
\end{cases}
    \end{multline}
This concludes the proof of \eqref{eq:boundpnab}.\\

We now turn to \eqref{eq:momentnabulog}. In light of \eqref{eq:momentnabu}, we may assume that $|x-y|\ll L$. By stationarity we may further  assume that $y=0$. 
Using the notation introduced in \eqref{eq:defY}, we have 
\[
 \nabla u_r^{1,L}(x)- \nabla u_R^{1,L}(0)=\sum_{i=1}^{L^d} (Y_{i,r}(x)-Y_{i,R}(0))
\]
so that once again by Rosenthal's inequality,
\[
 \EE\lt[ |\nabla u_r^{1,L}(x)- \nabla u_R^{1,L}(0)|^p\rt]^{\frac{1}{p}}\les \lt(L^d \EE[|Y_{i,r}(x)-Y_{i,R}(0)|^2]\rt)^{\frac{1}{2}} + \lt( L^d \EE[|Y_{i,r}(x)-Y_{i,R}(0)|^p]\rt)^{\frac{1}{p}}.
\]
For $p>2$, using triangle inequality and \eqref{eq:boundpnab} we have 
\begin{multline*}
 \lt( L^d \EE[|Y_{i,r}(x)-Y_{i,R}(0)|^p]\rt)^{\frac{1}{p}}\les \lt( L^d \EE[|Y_{i,r}(x)|^p]\rt)^{\frac{1}{p}}+\lt( L^d \EE[|Y_{i,R}(0)|^p]\rt)^{\frac{1}{p}}\\
 \les \frac{1}{r^{1-\frac{2}{p}}}+ \frac{1}{R^{1-\frac{2}{p}}}\les \frac{1}{r^{1-\frac{2}{p}}}.
\end{multline*}
So that we are left with the proof of 
\begin{equation}\label{eq:toprovelog2}
 \lt(L^d \EE[|Y_{i,r}(x)-Y_{i,R}(0)|^2]\rt)^{\frac{1}{2}}\les \log^{\frac{1}{2}}\lt(\frac{R+|x|}{r}\rt) +1.
\end{equation}
Notice first that since $X_i$ is uniformly distributed,
\[
  \lt(L^d \EE[|Y_{i,r}(x)-Y_{i,R}(0)|^2]\rt)^{\frac{1}{2}}=\lt( \int_{Q_L} \lt|\int_{Q_L} (\eta_r(y-x) -\eta_R(y) )\nabla G^L(X-y) dy\rt|^2 dX\rt)^{\frac{1}{2}}.
\]
By \eqref{eq:restGreen}, it is enough to prove \eqref{eq:toprovelog2} with $\chi(\cdot/L)\nabla G_{\R^d}$ instead of $\nabla G^L$. We now impose as a further constraint on $\chi$ that $\chi=1$ in $B_{\frac{1}{2}}$ so that 
\begin{align*}
 \lefteqn{\lt( \int_{Q_L} \lt|\int_{Q_L} (\eta_r(y-x) -\eta_R(y) )\chi\lt(\frac{X-y}{L}\rt)\nabla G_{\R^d}(X-y) dy\rt|^2 dX\rt)^{\frac{1}{2}}}\\
 &\les \lt( \int_{Q_L\backslash B_{\frac{L}{2}}} \lt(\int_{Q_L} |\eta_r(y-x) -\eta_R(y) |\frac{1}{|X-y|} dy\rt)^2 dX\rt)^{\frac{1}{2}}\\
 &\qquad + \lt( \int_{\R^d} \lt|\int_{Q_L} (\eta_r(y-x) -\eta_R(y) )\frac{X-y}{|X-y|^2} dy\rt|^2 dX\rt)^{\frac{1}{2}}.
\end{align*}
Since $\max(|x|+r, R)\ll L$, arguing as for  \eqref{eq:estimfar} we may bound  the first term by a constant and we are left with the estimate of the second term. We split it as follows:
\begin{multline}\label{eq:splitlog}
 \lt( \int_{\R^d} \lt|\int_{Q_L} (\eta_r(y-x) -\eta_R(y) )\frac{X-y}{|X-y|^2} dy\rt|^2 dX\rt)^{\frac{1}{2}}\\
 \les \lt( \int_{B_{3(R+|x|)}^c} \lt|\int_{Q_L} (\eta_r(y-x) -\eta_R(y) )\frac{X-y}{|X-y|^2} dy\rt|^2 dX\rt)^{\frac{1}{2}}\\
 + \lt( \int_{B_{3(R+|x|)}} \lt(\int_{Q_L} \eta_r(y-x)\frac{1}{|X-y|} dy\rt)^2 dX\rt)^{\frac{1}{2}} \\
  + \lt( \int_{B_{3(R+|x|)}} \lt(\int_{Q_L} \eta_R(y)\frac{1}{|X-y|} dy\rt)^2 dX\rt)^{\frac{1}{2}}.
\end{multline}
Arguing as for \eqref{eq:estimclose} and \eqref{eq:estimfar}, the last two terms are estimated as
\begin{multline*}
  \lt( \int_{B_{3(R+|x|)}} \lt(\int_{Q_L} \eta_r(y-x)\frac{1}{|X-y|} dy\rt)^2 dX\rt)^{\frac{1}{2}} 
  \\+ \lt( \int_{B_{3(R+|x|)}} \lt(\int_{Q_L} \eta_R(y)\frac{1}{|X-y|} dy\rt)^2 dX\rt)^{\frac{1}{2}}\\
  \les \log^{\frac{1}{2}}\lt(\frac{R+|x|}{r}\rt) + \log^{\frac{1}{2}}\lt( \frac{ R+|x|}{R}\rt)+1
  \les \log^{\frac{1}{2}}\lt(\frac{R+|x|}{r}\rt) +1.
\end{multline*}
We finally turn to the first term in \eqref{eq:splitlog}. Since on the one hand 
\[
 \int_{Q_L}  (\eta_r(y-x) -\eta_R(y) ) \frac{X}{|X|^2} dy=0
\]
and on the other hand, for $y\in B_{R+|x|}$ and $X\in B_{3(R+|x|)}^c$,
\[
 \lt|\frac{y-X}{|y-X|^2}-\frac{X}{|X|^2}\rt|\les \frac{|y|}{|X|^2},
\]
we find
\begin{multline*}
 \lt( \int_{B_{3(R+|x|)}^c} \lt|\int_{Q_L} (\eta_r(y-x) -\eta_R(y) )\frac{X-y}{|X-y|^2} dy\rt|^2 dX\rt)^{\frac{1}{2}}\\
 \les  \lt( \int_{B_{3(R+|x|)}^c} \lt(\int_{Q_L} |\eta_r(y-x) -\eta_R(y) |\frac{|y|}{|X|^2} dy\rt)^2 dX\rt)^{\frac{1}{2}}\\
 \les \lt( \int_{B_{3(R+|x|)}^c} \frac{(R+|x|)^2}{|X|^4} dX \rt)^{\frac{1}{2}}\les 1.
\end{multline*}
This concludes the proof of \eqref{eq:toprovelog2}.
\end{proof}

\begin{remark}
 Using Fourier series together with  combinatorial arguments instead of the Green kernel representation, bounds similar (albeit weaker) to \eqref{eq:momentnabu} have been obtained in the case $p=4$ in \cite{BobLe}.
\end{remark}

We now combine Lemma \ref{lem:boundnablu} with Theorem \ref{theo:mainnegsob} and Lemma \ref{lem:removesmallscale} to obtain moment estimates for $\nabla u^{R,L}$ in appropriate negative Sobolev spaces.
\begin{proposition}\label{prop:moments}
For every $d\ge 3$, $p\ge 2$, $\gamma>d(1-\frac{1}{p})-1$, $L\gg\ell\gg1$ and $R\gg1$,
\begin{equation}\label{eq:1boundmom}
 \frac{1}{|B_\ell|} \EE[\|\nabla u^{R,L}\|^p_{W^{-\gamma,p}(B_\ell)}]\les 1.
\end{equation}
Moreover, if $p\ge 2$ and $\gamma > \frac{d}{2}-1$, $L\gg\ell\gg1$ and $R\gg1$,
\begin{equation}\label{eq:2boundmom}
 \frac{1}{|B_\ell|} \EE[\|\nabla u^{R,L}_{\frac{1}{R}}\|^p_{W^{-\gamma,p}(B_\ell)}]\les 1.
\end{equation}
 If $d=2$, $p\ge 2$, $\gamma>1-\frac{2}{p}$, $L\gg\ell\gg1$ and $R\gg1$,
\begin{equation}\label{eq:3boundmom}
 \frac{1}{|B_\ell|} \EE[\|\nabla u^{R,L}-\nabla u^{R,L}_1(0)\|^p_{W^{-\gamma,p}(B_\ell)}]\les \log^{\frac{p}{2}} \ell.
\end{equation}
Finally, if $p\ge 2$ and $\gamma > 0$, $L\gg\ell\gg1$ and $R\gg1$,
\begin{equation}\label{eq:4boundmom}
 \frac{1}{|B_\ell|} \EE[\|\nabla u^{R,L}_{\frac{1}{R}} -\nabla u^{R,L}_1(0)\|^p_{W^{-\gamma,p}(B_\ell)}]\les \log^{\frac{p}{2}} \ell.
\end{equation}
\end{proposition}
\begin{proof}
 We start by noting that by Sobolev embedding (see Remark \ref{rem:Sob}), we may assume without loss of generality that $\gamma\notin \N$.  We will only prove \eqref{eq:1boundmom} and \eqref{eq:4boundmom} since the proof of \eqref{eq:2boundmom} and \eqref{eq:3boundmom} is very similar.\\
 We start with \eqref{eq:1boundmom}. Using \eqref{eq:main} we have 
 \begin{align*}
  \frac{1}{|B_\ell|} \EE[\|\nabla u^{R,L}\|^p_{W^{-\gamma,p}(B_\ell)}]&\les \frac{1}{|B_\ell|} \int_0^1 \eps^{p\gamma} \int_{B_{2\ell}} \EE[|\nabla u^{R,L}_\eps(x)|^p]  \frac{d\eps}{\eps}.
  \end{align*}
  Using the scaling relation \eqref{rescalingnabu} the estimate \eqref{eq:momentnabu} turns into
  $$ \EE\lt[ |\nabla u_\eps^{R,L}(x)|^p\rt]^{\frac{1}{p}}\les\lt( 1+ \frac{1}{(\eps R)^{\frac{d(p-2)}{2p}}}\rt)\frac{1}{\eps^{\frac{1}{2}(d-2)}} $$
so that we can estimate                                                           
  \begin{align*}
  \frac{1}{|B_\ell|} \EE[\|\nabla u^{R,L}\|^p_{W^{-\gamma,p}(B_\ell)}]&\les \int_0^1 \eps^{p\gamma} \lt( 1+ \frac{1}{(\eps R)^{\frac{d(p-2)}{2}}}\rt)\frac{1}{\eps^{\frac{p}{2}(d-2)}} \frac{d\eps}{\eps}\\
  &\les \frac{1}{R^{\frac{d(p-2)}{2}}}\int_0^{\frac{1}{R}} \eps^{p\gamma} 
  \frac{1}{\eps^{p\lt(d(1-\frac{1}{p})-1\rt)}} \frac{d\eps}{\eps} + \int_{\frac{1}{R}}^1 \eps^{p\gamma}  \frac{1}{\eps^{\frac{p}{2}(d-2)}} \frac{d\eps}{\eps}.
 \end{align*}
For the first integral to be finite we need $\gamma>d(1-\frac{1}{p})-1$ (which implies $\gamma>\frac{1}{2}(d-2)$ since $p\ge 2$ so that the second integral is also finite). Under this condition we indeed find that 
\[
 \frac{1}{R^{\frac{d(p-2)}{2}}}\int_0^{\frac{1}{R}} \eps^{p\gamma} 
  \frac{1}{\eps^{p\lt(d(1-\frac{1}{p})-1\rt)}} \frac{d\eps}{\eps} + \int_{\frac{1}{R}}^1 \eps^{p\gamma}  \frac{1}{\eps^{\frac{p}{2}(d-2)}} \frac{d\eps}{\eps}\les 1+ \frac{1}{R^{p( \gamma -\frac{1}{2}(d-2))}}\les 1
\]
and \eqref{eq:1boundmom} follows.\\
We now turn to \eqref{eq:4boundmom}. As above we use the scaling relation \eqref{rescalingnabu} and \eqref{eq:momentnabulog} to obtain
$$ \EE\lt[ |\nabla u_\eps^{R,L}(x)- \nabla u_1^{R,L}(y)|^p\rt]^{\frac{1}{p}}\les \lt( 1+ \frac{1}{(\eps R)^{\frac{p-2}{p}}}\rt)\lt( \log^{\frac{1}{2}}\lt(\frac{1+|x-y|}{\eps}\rt)+1\rt). $$
This estimate together with \eqref{eq:removesmallscale} from Lemma \ref{lem:removesmallscale} implies,
\begin{align*}
 \frac{1}{|B_\ell|} \EE[\|\nabla u^{R,L}_{\frac{1}{R}} -\nabla u^{R,L}_1(0)\|^p_{W^{-\gamma,p}(B_\ell)}]&\les \frac{1}{R^{p\gamma}|B_{\ell}|}\int_{B_{3\ell}}\EE[|\nabla u^{R,L}_{\frac{1}{R}}(x) -\nabla u^{R,L}_1(0)|^p]\\
 & \quad + \frac{1}{|B_\ell|}\int_{\frac{1}{R}}^1 \eps^{p\gamma}\int_{B_{3\ell}}\EE[|\nabla u^{R,L}_{\eps}(x) -\nabla u^{R,L}_1(0)|^p] \frac{d\eps}{\eps}\\
 &\les \frac{1}{R^{p\gamma}} \log^{\frac{p}{2}}( \ell R) + \int_{\frac{1}{R}}^1 \eps^{p\gamma} \log^{\frac{p}{2}}\lt(\frac{\ell}{\eps}\rt)\frac{d\eps}{\eps}\\
 &\les \log^{\frac{p}{2}} \ell.
\end{align*}
\end{proof}

\subsection{Convergence}\label{sec:convlinear}
We can now prove the main result of this section, namely the convergence of $\nabla u^{R,L}$ to the curl-free GFF.
\begin{theorem}\label{theo:convlinearbody}
 Let $p\ge 2$ and  $\gamma>d(1-\frac{1}{p})-1$. If $d\ge 3$ (respectively $d=2$), $\nabla u^{R,L}$ (respectively $\nabla u^{R,L}-\nabla u_1^{R,L}(0)$) converges in law in $W^{-\gamma,p}_\loc$ to $\nabla \Psi$ (respectively $\nabla \Psi-\nabla \Psi_1(0)$) as $R,L\to\infty$. \\
 Under the weaker assumption $\gamma> \frac{d}{2}-1$, the same conclusion holds with $\nabla u^{R,L}_{\frac{1}{R}}$ instead of $\nabla u^{R,L}$. Moreover, for every $\ell\ge 1$, all these random distributions have bounded moments of arbitrary order in $W^{-\gamma,p}(B_\ell)$.
\end{theorem}
\begin{proof}
 Let $\xi^{R,L}=\nabla u^{R,L}-\chi_{d=2}\nabla u_1^{R,L}$. By Proposition \ref{prop:moments}, and the compactness of the Sobolev embedding $W^{-\gamma,p}(B_\ell)\subset W^{-\gamma',q}(B_\ell) $ for $\gamma'<\gamma$ and $q\le p$, we have tightness of $\xi^{R,L}$ and $\xi^{R,L}_{\frac{1}{R}}$ under the above hypothesis on $\gamma$. By Prokhorov, this implies convergence in law up to subsequence. It thus only remains  to identify the limit.\\
 By density  and the Cramer-Wold lemma,
 it is enough to  prove that  for every $f\in C^\infty_c(B_\ell,\R^d)$, the real valued random variables $\xi^{R,L}(f)$ and $\xi^{R,L}_{\frac{1}{R}}(f)$ both converge in law  to $\nabla \Psi(f)-\chi_{d=2}\nabla\Psi_1(f)$. 
 We first note that
 \[
  |\xi^{R,L}(f)-\xi^{R,L}_{\frac{1}{R}}(f)|= |\xi^{R,L}(f-f_{\frac{1}{R}})|\le \|\xi^{R,L}\|_{W^{-\gamma,p}(B_{2\ell})} \|f-f_{\frac{1}{R}}\|_{W^{\gamma,p'}}
 \]
and $\|f-f_{\frac{1}{R}}\|_{W^{\gamma,p'}}\to 0$ as $R\to \infty$. Hence, using Proposition \ref{prop:moments} we see that it is enough to identify the limit of $\xi^{R,L}$. Moreover, letting 
\[
 \tilde{f}=\begin{cases}
          f & \textrm{ if } d\ge 3\\
          f- \int \eta f & \textrm{ if } d=2,
         \end{cases}
\]
we have using integration by parts, 
\[
 \xi^{R,L}(f)=\nabla u^{R,L}(\tilde{f})=W^{R,L}(\phi^{L})
\]
where $\phi^{L}$ is the $Q_L-$periodic solution with average zero of 
\[
 -\Delta \phi^{L}=\nabla \cdot \tilde{f} \qquad \textrm{ in } Q_L.
\]
Recalling \eqref{defnablapsi}, it is therefore enough to prove that $W^{R,L}(\phi^{L})$ converges in law to $\mathcal{N}\lt(0,\int (\phi^\infty)^2\rt)$ where $\phi^\infty$ is the $L^2(\R^d)$ solution of (cf.\ \eqref{defphi})
\[
 -\Delta \phi^\infty= \nabla \cdot \tilde{f} \qquad \textrm{ in } \R^d.
\]
For instance by Fourier analysis it follows that on the one hand $\int_{Q_L} (\phi^L)^2$ and $\int(\phi^\infty)^2$ are both finite using for $d=2$ that  $\int \tilde f=0$  (cf.\ discussion around \eqref{defphi2d}) and on the other hand that $\int_{Q_L} (\phi^L)^2 \to \int (\phi^\infty)^2$.
Moreover, by the Green function representation of $\phi^L$ it follows that $\|\phi^L\|_\infty\les 1$ (we stress that the constant is independent of $L\ge 1$).

We now mimic the classical proof of the central limit theorem and show that for every $\lambda>0$,
\begin{equation}\label{toproveconvtestreduced}
 \EE[\exp(\lambda W^{R,L}(\phi^{L}))]\to \exp\lt(\frac{\lambda}2\int(\phi^\infty)^2\rt).
\end{equation}
By definition of $W^{R,L}$ (see \eqref{def:WRL}) and the fact that $\int_{Q_{L}} \phi^{L}=0$, we have with  $N=(RL)^d$
\[
 W^{R,L}(\phi^{L})=R^{-d/2} \sum_{i=1}^{N} \phi^{L}(X_i)
\]
where $(X_i)_{i=1}^N$ are iid random variables uniformly distributed on $Q_L$. Therefore, by independence
\begin{align*}
 \EE[\exp(\lambda W^{R,L}(\phi^{L}))]=\lt(\frac{1}{|Q_L|}\int_{Q_{L}} \exp(\lambda R^{-d/2} \phi^{L})\rt)^{N}.
\end{align*} 
By the uniform boundedness of $\phi^{L}$, we have by Taylor expansion
\[
 \exp(\lambda R^{-d/2} \phi^{L})=1 + \lambda R^{-d/2} \phi^{L} + \frac{\lambda^2}{2} R^{-d} (\phi^{L})^2 + O_\lambda( R^{-3d/2}) (\phi^{L})^{3/2}.
\]
Integrating and using that $\int_{Q_{L}} \phi^{L}=0$, we get 
\[
 \EE[\exp(\lambda W^{R,L}(\phi^{L}))]=\lt(1 + \frac{\lambda^2}{2N }\int_{Q_{L}} (\phi^{L})^2 + \frac{1}{L^d} O_\lambda(R^{-d/2})\rt)^N.
\]
We conclude the proof of \eqref{toproveconvtestreduced} by sending $R,L\to \infty$.
\end{proof}

We close this section showing that in $d=2$ the shift $\nabla u^{R,L}_1(0)$ is itself quantitatively close to a Gaussian.

\begin{lemma}\label{lem:nablauCLT}
Let $\sigma^2 = \frac{1}{2} \EE[|\nabla u^{R,L}_1(0)|^2]$. Then $\sigma^2\sim \log L$
and for every  $p\geq 2$, 
$$W_p\lt(\mathsf{Law}\lt(\nabla u^{R,L}_1(0)\rt),\mathcal N\lt(0,\sigma^2 \mathsf{Id}\rt)\rt) \les \frac{1}{R \log^{\frac{1}{p}} L}.$$
\end{lemma}

\begin{proof}
If $X$ is a uniformly distributed random variable on $Q_L$ and $Y$ is given by (recall \eqref{eq:defY})
$$Y=(Y^{(1)},Y^{(2)})=\int \eta(y) \nabla G^L(X-y)dy,$$
then $\nabla u_1^{R,L}(0)=\frac{1}{R }\sum_{i=1}^{(RL)^2} Y_i $ for some iid copies $Y_i$ of $Y$ (cf.\ Section \ref{sec:momentbounds}). In order to apply  \cite[Th. 1]{Bo20}, 
let us first compute the covariance matrix of $Y$ (actually of $L Y$). We start by observing that  $\EE[Y]=0$ and by independence, $\EE[|Y|^2]= \frac{1}{L^2} \EE[|\nabla u^{R,L}_1(0)|^2]= \frac{2\sigma^2}{L^2}$.
Now by symmetry, we have that $Y^{(1)}$ and $Y^{(2)}$ have the same distribution and that the law of $Y$ is invariant under rotations preserving the cube so that in particular $(Y^{(1)},Y^{(2)})$ has the same law as 
$(-Y^{(2)},Y^{(1)})$. Therefore, $\EE[Y^{(1)}Y^{(2)}]=-\EE[Y^{(1)}Y^{(2)}]=0$ and $\EE[Y Y^T]=\frac{1}{2} \EE[|Y|^2] \mathsf{Id}= \frac{\sigma^2}{L^2} \mathsf{Id}$.\\
We now show that  $\sigma^2\sim \log L$.  Using the defining property of $G^L$ and $\int G^L=0$, we have indeed
\begin{align*}
 \sigma^2\sim L^2\EE[|Y|^2]
 &=\int\int\int\eta(y)\eta(z)\nabla G^L(X-z)\cdot \nabla G^L(X-y) dydzdX\\
 &= -\int\int \eta(y)\eta(z) G^L(y-z)\\
 &= - \int\int \eta(y)\eta(z) G^1\lt(\frac{y-z}{L}\rt) \sim \log(L),
\end{align*}
 Notice that the upper bound  $\sigma^2\les \log L$ was actually already proven in \eqref{eq:boundpnab}.\\

Setting  $Z_i= \sigma^{-1} L Y_i$, we have  $\EE[Z_i Z_i^T]=\mathsf{Id}$ and $\nabla u_1^{R,L}(0)= \frac{\sigma}{R L} \sum_{i=1}^{(RL)^2} Z_i$.   By \cite[Th. 1]{Bo20} and Cauchy-Schwarz, we have
\begin{align*} 
W_p\lt(\mathsf{Law}\lt(\frac{1}{RL}\sum_{i=1}^{(RL)^2}Z_i\rt),\mathcal N(0,\mathsf{Id})\rt) &\les \frac{1}{RL}\lt( \nor{\EE\lt[Z_1Z_1^T|Z_1|^2\rt]}_{HS}^\frac12 + \EE\lt[|Z_1|^{p+2}\rt]^\frac1p \rt) \\
&= \frac{1}{RL}\lt( \frac{L^2}{\sigma^2} \nor{\EE\lt[YY^T|Y|^2\rt]}_{HS}^\frac12+  \lt(\frac{L}{\sigma}\rt)^{\frac{p+2}{p}}\EE\lt[|Y|^{p+2}\rt]^\frac1p \rt) \\
&\les \frac{1}{R}\lt( \frac{1}{\sigma^2} \EE\lt[ L^2|Y|^4\rt]^\frac12+  \frac{1}{\sigma^{\frac{p+2}{p}}}\EE\lt[ L^2 |Y|^{p+2}\rt]^\frac1p \rt)\\
&\stackrel{\eqref{eq:boundpnab}}{\les} \frac{1}{R}\lt(\frac{1}{\sigma^2} + \frac{1}{\sigma^{\frac{p+2}{p}}} \rt)\les \frac{1}{R \log^{\frac{p+2}{2p}} L }.
\end{align*}
Using  $W_p(\mathsf{Law}(aU),\mathsf{Law}(aV))= a W_p(\mathsf{Law}(U),\mathsf{Law}(V))$ for two random variables $U,V$ and $a>0$ this implies 
$$W_p\lt(\mathsf{Law}\lt(\frac{1}{R}\sum_{i=1}^{(RL)^2}Y_i\rt),\mathcal N(0,\sigma^2\mathsf{Id})\rt) \les \frac{\sigma}{R \log^{\frac{p+2}{2p}} L } \les  \frac{1}{R\log^{\frac1p}(L)}. $$
\end{proof}

\section{Convergence of the displacement}\label{sec:conv}
\subsection{Deterministic input}\label{sec:deter}
We first need to adapt the main result of \cite{GHO} to the periodic setting. To this end, we define the following rate function encoding the speed of convergence for the Euclidean matching problem
\begin{equation}\label{eq:defbeta}
\beta(r)=\begin{cases}
           1 & d\geq 3\\
           \log(1+r) & d=2
          \end{cases}.
\end{equation}
However, we would like to stress that the following theorem also holds for more general rate functions satisfying the assumptions of \cite{GHO}.   
\begin{theorem}\label{theo:main}
 Let $L\gg r_*\gg1$ and  $\mu$ be a  $Q_L-$periodic measure  such that $\mu(Q_L)=L^d$. Let $u$ be
  a $Q_L-$periodic   
  solution of the Poisson equation 
\begin{equation}\label{ma87}
 \Delta u=\mu-1 \ \textrm{ in } \R^d
\end{equation}
and define for every $r>0$,
\begin{equation}\label{defhR}
 h_r=\int \eta_r \nabla u.
\end{equation}
Assume that\footnote{notice that $W_{per}^2(\mu,1)\le W^2_{Q_L}(\mu,1)$}  
 \begin{equation}\label{hypstart}
 \frac{1}{L^d} W^2_{Q_L}(\mu,1)\le r_*^2\beta\lt(\frac{L}{r_*}\rt)
 \end{equation}
and (recall  notation \ref{notationWkappa})
\begin{equation}\label{hypdata2}
 \frac{1}{|Q_r|} W^2_{Q_r}\lt(\mu, \kappa\rt)\le r_*^2 \beta\lt(\frac{r}{r_*}\rt) \qquad \forall r\in[ r_*, L], \textrm{ dyadic},
\end{equation}
 hold. Then, letting $\pi$ be the $Q_L-$periodic optimal transport plan for $W_{per}(\mu,1)$ we have for every $ r \ges r_*$,  
\begin{equation}\label{mainweaktheo}
 \lt|\int_{\R^d\times \R^d} \eta_r(x)(y-x-h_r)d\pi\rt|\les  r_*^2\frac{\beta\lt(\frac{r}{r_*}\rt)}{r} 
\end{equation}
and 
\begin{multline}\label{strongestimintro}
 \sup \lt\{ |x-(y-h_r)| \ : \ (x,y)\in \spt\pi\cap ((B_{r}\times\R^d)\cup(\R^d\times B_{r}(h_r))\rt\}\\
 \les  \,r \lt(\frac{r_*^2 \beta\lt(\frac{r}{r_*}\rt)}{r^2}\rt)^{\frac{1}{d+2}}.
\end{multline}
\end{theorem}
\begin{proof}
 As noticed in \cite[Rem. 1.3]{GHO}, by scaling we may assume that $r_*=1$. The idea of the proof is now to apply \cite[Th. 1.2]{GHO} on very large balls. Let us start with a few preliminary observations and set up notation.\\
 If $\pi$ is the optimal transport plan for $W_{per}(\mu,1)$, we can identify it with a transport plan on $\R^d\times\R^d$ which is $Q_L-$periodic  and we have 
\[
 W^2_{per}(\mu,1)=\int_{Q_L\times\R^d} |x-y|^2 d\pi =\int_{\R^d\times Q_L}|x-y|^2 d\pi.
\]
Let us point out a consequence of this fact. Since the Lebesgue measure is invariant under translation, for any $\xi\in \R^d$, the plan $\pi(x,y-\xi)$ is admissible and thus by minimality,
\[
  W^2_{per}(\mu,1)=\min_{\xi\in \R^d}\int_{Q_L\times\R^d} |x-y+\xi|^2 d\pi.
\]
Minimizing in $\xi$ we find as expected that 
\begin{equation}\label{averagezero}
 \int_{Q_L\times\R^d} (x-y)d\pi=0. 
\end{equation}
Let us also point out that by the characterization of minimality for $W_{per}$, the coupling $\pi$ has cyclically monotone support on $\R^d$. 
In particular, defining for $r>0$, the coupling $\pi^r$ by setting   $\pi^r(\Omega_1\times \Omega_2)=\pi( (\Omega_1\cap Q_r)\times \Omega_2)$, 
this is the optimal transport plan between $\mu\restr Q_r$ and $\lambda_r= \pi^r_2$. If now $r$ is a multiple of $L$, we have by periodicity of $\pi$,
\begin{equation}\label{hypstartbigball}
\frac{1}{r^d}\int |x-y|^2 d\pi^r=\frac{1}{L^d} \int_{Q_L\times \R^d} |x-y|^2 d\pi\stackrel{\eqref{hypstart}}{\le} \beta(L).\end{equation} 
For $t\in [0,1]$, we define the flux-density pair $(\rho,j)$ as 
\begin{equation}\label{BB}
 \int \xi\cdot dj_t=\int \xi((1-t)x+ty)\cdot (y-x) d\pi \qquad \textrm{and} \qquad \int \zeta d\rho_t=\int \zeta((1-t)x+ty) d\pi
\end{equation}
so that both $j$ and $\rho$ are $Q_L-$periodic (we see them here as defined in $\R^d$). We then let $(\bar \rho,\bar j)=\int_0^1 (\rho_t,j_t)$ be their integrals in time.\\

\medskip
{\textit Step 1: Application of \cite[Th. 1.2]{GHO}.}\\
Let 
\[\tilde{\beta}(r)=\begin{cases}
                    \beta(r) &\textrm{ if } r\le L\\
                    \beta(L) & \textrm{ if } r\ge L.
                   \end{cases}
\]

We claim that if either  $r$ is  dyadic with $L\ge r\ge 1$ or $r\in L\N$, then
\begin{equation}\label{hypdataRd}
 \frac{1}{r^d}W_{Q_r}^2(\mu,\kappa)\le \tilde{\beta}(r).
\end{equation}
 By \eqref{hypdata2}, there is nothing to prove for $r\le L$. 
Now if $r\in L\N$, we may decompose $Q_r$ in boxes which are translates of $Q_L$ and use as competitor (on $Q_r$) a concatenation of 
$\frac{|Q_r|}{|Q_L|}$ translates of $\hat{\pi}$, where $\hat{\pi}$ is the (Euclidean) optimal transport plan for $W_{Q_L}(\mu,1)$ and obtain (in this case $\kappa=1$)
\[
 \frac{1}{r^d}W_{Q_r}(\mu,1)\le \frac{1}{L^d} W_{Q_L}(\mu,1)\le \beta(L)=\tilde{\beta}(r).
\]
Using \cite[Lem. 2.10 \& Rem. 2.11]{GHO}, we may find a sequence of approximately geometrical radii $r_k\ge 1$ (with $r_k\to \infty$ as $k\to \infty$) for which 
\[
 \frac{1}{|B_{r_k}|} W_{B_{r_k}}^2(\mu,\kappa)\les \tilde{\beta}(r_k).
\]
By periodicity of $\pi_r$, we see that for every $r\gg L$, there exists $\bar r\le r$ with $r\sim \bar r$ such that $\lambda_r\restr B_{\bar r}= dx\restr B_{\bar r}$ (and moreover if $(\rho^r,j^r)$ 
are defined as in \eqref{BB} with $\pi$ replaced by $\pi^r$, $(\rho^r,j^r)=(\rho,j)$ in $ B_{\bar r}\times [0,1]$).  
We are thus in a position to apply \cite[Th. 1.2]{GHO}  to find a diverging sequence $\bar r_k$ such that if we define $u_k$ as the mean zero solution of (here $\nu$ denotes the outward normal to $B_{\bar r_k}$)
\begin{equation}\label{eq:uk}
 \Delta u_k=\mu-1 \textrm{ in } B_{\bar r_k} \qquad \textrm{ and } \qquad \nu\cdot \nabla u_k= \bar j\cdot \nu \textrm{ on } \partial B_{\bar r_k}
\end{equation}
and such that 
defining
\[
 h_r^k=\int \eta_r \nabla u_k,
\]
then for $\bar r_k\ge r\ges 1$,
\begin{equation}\label{mainweakk}
 \lt|\int_{\R^d\times \R^d} \eta_r(x)(y-x-h_r^k)d\pi\rt|\les \frac{\tilde{\beta}(r)}{r} 
\end{equation}
and 
\begin{multline}\label{strongestimk}
 \sup \lt\{ |x-(y-h^k_r)| \ : \ (x,y)\in \spt\pi\cap ((B_{r}\times\R^d)\cup(\R^d\times B_{r}(h^k_r))\rt\}
 \les  \,r \lt(\frac{\tilde{\beta}(r)}{r^2}\rt)^{\frac{1}{d+2}}.
\end{multline}
Before proceeding further, let us recall that the choice of $\bar r_k$ is such that
(see \cite[Prop 3.6, (3.47)]{GHO})\footnote{Here we use with the notation of \cite{GHO}, $\int_{\partial B_{\bar r_k}} |\bar j\cdot \nu|=\int_{\partial B_{\bar r_k}} |\hat{g}|$.}
\begin{equation}\label{boundfluxj}
 \frac{1}{{\bar r}_k^{d-1}}\int_{\partial B_{\bar r_k}} |\bar j\cdot \nu|
 \les \tilde{\beta}^{1/2}(\bar r_{k}).
\end{equation}
Since by \eqref{eq:sizenabu} below, the solution $u$ of \eqref{ma87} satisfies
\begin{equation}\label{estimuQL}
 \frac{1}{L^d}\int_{Q_L}|\nabla u|\les L,
\end{equation}
by periodicity of $u$, we also  have for $r\gg L$ that 
\[\frac{1}{|B_r|}\int_{B_r} |\nabla u|\les L\]
and we can thus assume that $\bar r_k$ was chosen so that
\begin{equation}\label{boundfluxu}
 \frac{1}{\bar r_k^{d-1}}\int_{\partial B_{\bar r_k}} |\nabla u\cdot \nu|\les L.
\end{equation}
\medskip
{\textit Step 2.}
We now argue that $\nabla u_k\to \nabla u$ in $C^{\infty}_{loc}(\R^d)$ as $k\to \infty$. Passing to the limit in \eqref{mainweakk} and \eqref{strongestimk}, this would yield \eqref{mainweaktheo} and \eqref{strongestimintro}.\\
Let $\phi_k=u_k-u$. As a first step, we will show that $\nabla \phi_k$ is bounded. By harmonicity, this will imply that $\nabla \phi_k\to \nabla \phi$ in
$C^{\infty}_{loc}(\R^d)$ for some harmonic  and bounded function. By Liouville's theorem, this will imply that $\nabla \phi=\xi$ is constant.

To show boundedness of $\nabla \phi_k$ observe that by harmonicity we have for every $r>0$ and every $k$ large enough
\[
 \sup_{B_r}|\nabla \phi_k|\les \frac{1}{{\bar r}_k^d}\int_{B_{{\bar r}_k}}|\nabla \phi_k|.
\]
Arguing as in the proof of Lemma \ref{lem:nabu} below we can further estimate

\[
\frac{1}{{\bar r}_k^d}\int_{B_{{\bar r}_k}}|\nabla \phi_k|\les \frac{1}{{\bar r}_k^{d-1}}\int_{\partial B_{{\bar r}_k}}|\nu\cdot \nabla \phi_k|\le \frac{1}{{\bar r}_k^{d-1}}\int_{\partial B_{{\bar r}_k}}|\nu\cdot \nabla u_k|+ \frac{1}{{\bar r}_k^{d-1}}\int_{\partial B_{{\bar r}_k}}|\nu\cdot \nabla u|.
\]
Using \eqref{boundfluxj} and \eqref{boundfluxu} (and the fact that $\tilde{\beta}({\bar r}_k)=\beta(L)$), shows that $\nabla \phi_k$ is bounded and converges to the constant $\xi$.\\

We finally want to prove that $\xi=0$. By harmonicity of $\nabla \phi_k$, 
\[
 \xi=\lim_{k\to \infty} \nabla \phi_k(0)=\lim_{k\to\infty} \frac{1}{|B_{{\bar r}_k}|}\int_{B_{{\bar r}_k}} \nabla \phi_k.
\]
To conclude the proof it is enough to show that both $ \frac{1}{|B_{{\bar r}_k}|}\int_{B_{{\bar r}_k}} \nabla u_k$ and $\frac{1}{|B_{{\bar r}_k}|}\int_{B_{{\bar r}_k}} \nabla u$ tend to zero. Let us start with the easier term. Since $u$ is $Q_L-$periodic, we have 
\[
 \int_{Q_L}\nabla u=0.
\]
Covering $B_{{\bar r}_k-C L}$ with translates of $Q_L$ we have 
\begin{align*}
 \lt|\int_{B_{{\bar r}_k}} \nabla u\rt|\le\int_{B_{{\bar r}_k}\backslash B_{{\bar r}_k-CL}} |\nabla u|\les  L {\bar r}_k^{d-1} \int_{Q_L}|\nabla u|\stackrel{\eqref{estimuQL}}{\les} L^{d+1} {\bar r}_k^{d-1}
\end{align*}
from which 
\[\lt|\frac{1}{|B_{{\bar r}_k}|}\int_{B_{{\bar r}_k}} \nabla u\rt|\les \frac{L^{d+1}}{{\bar r}_k}\stackrel{k\to \infty}{\to}0.\]
We now turn to the second term. Since $\nabla u_k-\bar j$ has divergence zero in $B_{{\bar r}_k}$ with $\nu\cdot (\nabla u_k-\bar j)=0$ on $\partial B_{{\bar r}_k}$, 
\[
 \int_{B_{{\bar r}_k}} \nabla u_k=\int_{B_{{\bar r}_k}} \bar j.
\]
Let $M_k=\lt(\tilde{\beta}({\bar r}_k)/{\bar r}_k^2\rt)^{\frac{1}{d+2}}=\lt(\tilde{\beta}(L)/{\bar r}_k^2\rt)^{\frac{1}{d+2}}$. By the $L^\infty$ bound \cite[Lem 2.9 (2.32)]{ GHO}, we have for every $(x,y)\in \spt \pi \cap B_{{\bar r}_k}\times B_{{\bar r}_k}$,
\begin{equation}\label{Linfty}
 |x-y|\les {\bar r}_k M_k.  
\end{equation}
Notice that $M_k\to 0$. For $z\in (L\Z)^d$, let $Q^z_L=Q_L+z$. From \eqref{Linfty}, we deduce that if $Q_L^z\subset B_{(1-M_k){\bar r}_k}$ then for every $(x,y)\in (Q_{L}^z\times \R^d)\cap \spt \pi$, and every $t\in[0,1]$, $(1-t)x+ty\in B_{{\bar r}_k}$.
We deduce that 
\begin{align*}
 \int_{B_{{\bar r}_k}} \bar j&=\sum_{z\in (L/\Z)^d} \int_0^1\int_{Q_{L}^z\times \R^d} \chi_{B_{{\bar r}_k}}((1-t)x+ty) (y-x) d\pi dt\\
 &=\sum_{z\in (L/\Z)^d, Q_L^z\subset B_{(1-M_k){\bar r}_k} } \int_{Q_{L}^z\times \R^d} (y-x) d\pi \\
 &\qquad +\sum_{z\in (L/\Z)^d, Q_L^z\nsubseteq B_{(1-M_k){\bar r}_k} } \int_0^1\int_{Q_{L}^z\times \R^d} \chi_{B_{{\bar r}_k}}((1-t)x+ty) (y-x) d\pi dt\\
 &\stackrel{\eqref{averagezero}}{=}\sum_{z\in (L/\Z)^d, Q_L^z\nsubseteq B_{(1-M_k){\bar r}_k} } \int_0^1\int_{Q_{L}^z\times \R^d} \chi_{B_{{\bar r}_k}}((1-t)x+ty) (y-x) d\pi dt.
\end{align*}
We conclude that 
\[
 \lt|\frac{1}{|B_{{\bar r}_k}|} \int_{B_{{\bar r}_k}} \bar j\rt| \les \frac{M_k}{L^d} \int_{Q_L\times \R^d} |y-x| d\pi\les M_k \beta^{1/2}(L)\to 0.
\]
This shows that $\xi=0$ and thus putting  all together, we proved that $\nabla u_k\to \nabla u$ in $C^{\infty}_{loc}(\R^d)$ as $k\to \infty$.
\end{proof}

\begin{lemma}\label{lem:nabu}
 Let $L\gg 1$ and $\mu$ be a $Q_L$-periodic measure such that $\mu(Q_L)=L^d.$ Let $u$ be the $Q_L$-periodic 
 solution with zero average to 
 $$ \Delta u=\mu-1 \text{ in } \R^d.$$
 Then, for $p\in \lt[1,\frac{d}{d-1}\rt)$
 \begin{align}\label{eq:sizenabu}
  \lt(\frac{1}{L^d}\int_{Q_L} |\nabla u|^p\rt)^{\frac1p} \les L .
 \end{align}
\end{lemma}
\begin{proof}
 We fix $p\in(1,\frac{d}{d-1})$ so that its dual $q=\frac{p}{p-1}>d$. We will argue by the dual representation of the $L^p$-norm of $\nabla u$. To this end, let $\psi\in L^q(Q_L)$ be given and let $\omega$ be the $Q_L$-periodic solution with average zero of the dual problem
 $$\Delta \omega=\nabla\cdot \psi.$$
 We then have by integration by parts
\[
 \frac{1}{L^d}\int_{Q_L} \nabla u \cdot \psi =- \frac{1}{L^d}\int_{Q_L} \omega d(\mu-1)\les \lt(\frac{\mu(Q_L)}{L^d}+1\rt) \sup_{Q_L} |\omega|=2 \sup_{Q_L} |\omega|.  
\]
Now by Sobolev embedding (recall that $q>d$) and the Poincar\'e inequality,
\[
 \sup_{Q_L} |\omega| \les_q L \lt(\frac{1}{L^d}\int_{Q_L} |\nabla \omega|^q\rt)^{\frac{1}{q}}.
\]
Using finally Calderon-Zygmund estimates we have 
\[
 \lt(\frac{1}{L^d}\int_{Q_L} |\nabla \omega|^q\rt)^{\frac{1}{q}}\les_q \lt(\frac{1}{L^d}\int_{Q_L} |\psi|^q\rt)^{\frac{1}{q}}
\]
leading to 
\[
 \lt(\frac{1}{L^d}\int_{Q_L} |\nabla u|^p\rt)^{\frac{1}{p}}\les_p L .
\]
Finally, the case $p=1$ follows by H\"older's inequality.
\end{proof}

We close this section with a Lemma estimating $|h_r-h_{r'}|$. This is very similar to \cite[Lem. 4.3]{GHO} so we only sketch the proof.
\begin{lemma}\label{lem:difh}
 For $0<r\le r'$, let  $u$ be a solution of  
 \[
  \Delta u=\mu-1 \qquad \textrm{in } B_{r'}.
 \]
Then, for every $\alpha\in(0,1)$
\[
\lt|\int (\eta_r-\eta_{r'})\nabla u\rt|^2\les_\alpha\lt(\frac{r'}{r}\rt)^{2(d+\alpha)} \frac{1}{|B_{r'}|}W^2_{B_{r'}}(\mu,\kappa).
\]

\end{lemma}
\begin{proof}
 By scaling we may assume that $r'=1$. We now argue exactly as in \cite[Lem. 4.3]{GHO} and  consider the solution of 
\[
 \Delta \omega= \partial_1 (\eta_1-\eta_{r}) \quad \textrm{ in } B_1, \qquad \textrm{and } \qquad \omega=\nu\cdot\nabla \omega=0 \quad \textrm{ on } \partial B_1 
\]
and remark that by Schauder estimates, $[\omega]_{C^{1,\alpha}(B_1)}\les_\alpha [\eta_1-\eta_r]_{C^{0,\alpha}(B_1)}\les \frac{1}{r^{d+\alpha}}$.
\end{proof}

\subsection{Stochastic input}\label{sec:sto}

In this section, we will show that there is a random radius $r_{*,L}$ with stretched exponential moments independent of $L$ such that the assumptions of Theorem \ref{theo:main} are satisfied with $\mu=\mu^{1,L}$ and $r_*=r_{*,L}$.

To this end, recall the notation $Q_r=[-\frac{r}{2},\frac{r}{2})^d$ and  write $\mu^{1,L}=\mu$. 
We aim to control $
 W_{Q_r}^2\lt(\mu,\kappa\rt)$
for $1\ll r <\infty.$ This will be a rather direct consequence of \cite{AmStTr16} for $d=2$ and \cite{GolTre} for $d\geq 3.$
To state their result and to adapt it to our setup we need to introduce some notation. 

For $n\leq L$ we define the probability measure $\P_n$ with associated expectation operator $\EE_n$ by
\[
   \PP_n\sqa{F}= \frac{\PP\sqa{F\cap\{ \mu(Q_r)=n\}}}{\PP\sqa{\mu(Q_r)=n}}, 
  \]
Since $\mu=\sum_{i=1}^{L^d}\delta_{X_i}$ for iid uniform random variables $(X_i)_i$ it follows that $\mu(Q_r)$ is Binomially distributed with parameter $(L^d,r^d)$ so that
\begin{align}\label{eq:pn}
 p_n=\PP\sqa{\mu(Q_r)=n}={L^d\choose n} \lt(\frac{r^d}{L^d} \rt)^n \lt(1- \frac{r^d}{L^d} \rt)^{L^d-n}
\end{align}
which converges for $L\to \infty$ to $\exp(-r^d) \frac{r^{dn}}{n!}.$

Equipped with this probability measure, $\mu\restr Q_r$ can be identified with $n$ uniformly iid random variables $X_i$ on $Q_r$. Recall $\beta$ from \eqref{eq:defbeta}.
A simple rescaling shows that
\begin{align*}
 \frac{1}{r^2 n^{1-\frac{2}{d}} \beta(n)} \EE_n\sqa{W_{Q_r}^2\bra{\mu,\frac{\mu(Q_r)}{r^d} }}
 \end{align*}
is independent of $r$.
Arguing as in \cite[Rem. 4.7]{AmStTr16} and \cite[Rem. 6.5]{GolTre} and using the fact that the uniform measure on $[0,1]$ satisfies a log-Sobolev inequality to replace exponential bounds by Gaussian bounds,
it follows that  there exists $c_0>0$ such that for every $M \gg 1$ (since we pass from a deviation to a tail estimate) and $n$ large enough uniformly\footnote{Notice that the left-hand side of \eqref{eq:conc} actually does not depend on $r$.} in $r$,

\begin{equation}\label{eq:conc}
 \PP_n\sqa{\frac{1}{r^2 n^{1-\frac{2}{d}}\beta(n)} W_{Q_r}^2\lt(\mu, \frac{\mu(Q_r)}{r^d}\rt)\ge M}\le \exp(-c_0 M n^{1-\frac{2}{d}}\beta(n)).
\end{equation}

We now show how \eqref{eq:conc} can be turned into the desired control of $W_{Q_r}^2\lt(\mu,\kappa\rt)$. 

\begin{lemma}\label{lem:concentration}
For each $L\geq 1$ let $\mu^{1,L}$ be a Binomial point process of intensity 1 on $Q_L$. Then, 
there exists a universal constant $c>0$ (independent of $L,r$ and $M$) such that for $r\ge 1, M \gg 1$,
\begin{align}\label{eq:concentration}
 \PP\sqa{\frac{1}{ r^{d}\beta(r^d)} W_{Q_r}^2\left(\mu, \frac{\mu(Q_r)}{r^d}\right) \geq M} \leq \exp(-c M r^{d-2}\beta(r^d))\ . 
\end{align}
\end{lemma}

\begin{proof} 
We start by noting that by Chernoff bounds  there exists a constant $c$ independent of $L$ such that
\begin{equation}\label{chernoff}
 \PP\sqa{\frac{\mu(Q_r)}{r^d}\notin \sqa{\frac{1}{2},2}}\le \exp(-c r^d),
\end{equation}
and for $M\gg 1$ a (different) constant $c$ independent of $L$ such that
\begin{equation}\label{chernoff2}
 \PP\sqa{\frac{\mu(Q_r)}{r^d}\ge M}\le \exp(-c r^d M).
\end{equation}
 
For $1\le M\les \frac{r^d}{r^{d-2}\beta(r^d)}=\frac{r^2}{\beta(r^d)}$, by definition of $\PP_n$ 
\begin{align*}
 \lefteqn{\PP\sqa{\frac{1}{r^d \beta(r^d)} W_{Q_r}^2\left(\mu, \frac{\mu(Q_r)}{r^d}\right) \geq M}}\\
 &\le \PP\sqa{\frac{\mu(Q_r)}{r^d}\notin\sqa{\frac{1}{2},2}}+\sum_{n\in[r^d/2,2 r^d]} p_n \PP_n\sqa{\frac{1}{r^d \beta(r^d)} W_{Q_r}^2\left(\mu, \frac{\mu(Q_r)}{r^d}\right) \geq M}\\
 &\stackrel{\eqref{chernoff}}{\le}  \exp(-cr^d) + \sum_{n\in[r^d/2,2 r^d]} p_n\PP_n\sqa{\frac{1}{r^2 n^{1-\frac{2}{d}} \beta(n)} W_{Q_r}^2\left(\mu, \frac{\mu(Q_r)}{r^d}\right) \geq \frac{r^{d-2}\beta(r^d)}{ n^{1-\frac{2}{d}}\beta(n)}M}\\
 &\stackrel{\eqref{eq:conc}}{\le}\exp(-cr^d) + \sum_{n\in[r^d/2,2 r^d]} p_n \exp(-c_0 M r^{d-2}\beta(r^d))\\
 &\le \exp(-cr^d) +\exp(-c_0 M r^{d-2}\beta(r^d)) \le \exp(-c M r^{d-2}\beta(r^d)),
\end{align*}
while for $M\gg \frac{r^2}{\beta(r^d)}$, using the estimate $W_{Q_r}^2\lt(\mu, \frac{\mu(Q_r)}{r^d}\rt)\le  d r^2 \mu(Q_r)$ together with \eqref{chernoff2}, we obtain
\[
 \PP\sqa{\frac{1}{r^d \beta(r^d)} W_{Q_r}^2\left(\mu, \frac{\mu(Q_r)}{r^d}\right) \geq M}\le \PP\sqa{\frac{\mu(Q_r)}{r^d}\ge  \frac{M}{d}  \frac{\beta(r^d)}{r^2}}\le \exp(-c M r^{d-2}\beta(r^d)).
\]
\end{proof}

By giving up a bit of integrability we can strengthen \eqref{eq:concentration} into a supremum bound. 

\begin{theorem}\label{thm:mainstoch}
 For each $L\geq 1$ let $\mu^{1,L}$ be a Binomial point process on $Q_L$ of intensity one. Then there exists a universal constant $c>0$ and a family of random variables $(r_{*,L})_L$ such that 
 $$\sup_L \EE\lt[\exp\lt(c \frac{r^2_{*,L}}{\beta(r_{*,L})} \rt)\rt]<\infty$$
 and for every dyadic $r$ with $r_{*,L}\leq r\leq L$ 
  \begin{equation}\label{eq:concentrationrstar}
  \frac{1}{r^d}W_{Q_r}^2\lt(\mu, \frac{\mu(Q_r)}{r^d}\rt)\le r_{*,L}^2 \beta\lt(\frac{r}{r_{*,L}}\rt).
 \end{equation}
\end{theorem}
\begin{proof}
 We first prove that there exist a constant $\bar c>0$ independent of $L$ and for each $L\geq 1$ a    random variable $\Theta^L$  with $\textstyle\sup_L\EE\sqa{\exp(\bar c\Theta^L)}<\infty$
  such that for all dyadic   $r$ with $L\ge r\ge 1$,
\begin{equation}\label{eq:concentrationtheta}
  \frac{1}{r^d \beta(r^d)}  W_{Q_r}^2\left(\mu, \frac{\mu(Q_r)}{r^d}\right)\leq \Theta^L  .
\end{equation}
For  $ k\ge 1$, let $r_k= 2^k$ and put 
\begin{equation}\label{def:theta}
\Theta^L_k= \frac{1}{r_k^d \beta(r_k^d)} W_{Q_{r_k}}^2\left(\mu, \frac{\mu(Q_{r_k})}{r_k^d}\right), \quad \Theta^L=\sup_{ r_k\le L}\Theta^L_k.
\end{equation}
We claim that the exponential moments of $\Theta^L_k$ given by Lemma \ref{lem:concentration} translate into exponential moments for $\Theta^L$. Indeed fix  $1\gg \bar c>0$. Then, we estimate for some fixed $m\gg 1$
\begin{align*}
 \EE[\exp(\bar c \Theta^L)] &\leq \exp(\bar c m) + \sum_{ M\ge m}  \PP[\Theta^L\geq M]\exp(\bar c (M+1)) \\
& \leq \exp(\bar cm ) + \sum_{M\ge m} \exp(\bar c (M+1)) \sum_{r_k\le L} \PP[\Theta^L_k \geq M] \\
& \stackrel{\eqref{eq:concentration}}{\leq} \exp(\bar c m ) +  \exp(\bar c)\sum_{k\geq 1} \sum_{M\ge m}  \exp(-M(cr_k^{d-2}\beta(r_k^d)-\bar c))\\
& \stackrel{r_k=2^k \& \, \bar c\ll1}{\lesssim} \exp(\bar cm ) + \exp(\bar c) \sum_{k\geq 1} \exp( -  c  k) <\infty.
\end{align*}
Hence, $\Theta^L$ has exponential moments and \eqref{eq:concentrationtheta} is satisfied.
Moreover, the bound on $ \EE[\exp(\bar c \Theta^L)]$ is independ of $L$ so that also $\sup_L  \EE[\exp(\bar c \Theta^L)]<\infty.$

Without loss of generality, we may assume that  $\Theta^L\geq 1$. For $d\geq 3$ we set $r_{*,L}^2=\Theta^L.$
For $d=2$ we observe that $\beta(r^d)\sim \beta(r)$ for $r\geq 1$ and define $r_{*,L} \ge  1$  via the equation 
 \begin{equation}\label{def:rstar} 
\frac{r_{*,L}^2}{\beta(r_{*,L})}  = \frac{\Theta^L}{\beta(1)} ,
 \end{equation}
which has a solution since $r_{*,L}\mapsto r_{*,L}^2/\beta\lt(r_{*,L}\rt)$ is monotone on $(0,\infty)$. 
Since $r\mapsto  \beta\lt(\frac{r}{r_{*,L}}\rt)/\beta( r)$ is a non decreasing function, we have for $r\ge  r_{*,L}$,
\[
 \beta( r)\le \frac{ \beta(r_{*,L})}{ \beta(1) } \beta\lt(\frac{r}{r_{*,L}}\rt).
\]
This together with \eqref{def:rstar} gives for every dyadic $r\ge  r_{*,L}$,
\[
 \Theta^L  \beta( r)\le r_{*,L}^2 \beta\lt(\frac{r}{r_{*,L}}\rt),
\]
from which we see that \eqref{eq:concentrationtheta} implies \eqref{eq:concentrationrstar}. 
\end{proof}

\begin{remark}
 We defined the random radius $r_{*,L}$ for the Binomial point process $\mu^{1,L}$ of intensity one. Similarly, we could define a random radius $r_{*,L}^R$ for a Binomial point process $\mu^{R,L}$ of intensity $R^d$. They are connected by the scaling relation 
 $$r_{*,RL} = R r_{*,L}^R.$$
\end{remark}

After this preparation we obtain our main quenched result:

\begin{proposition}\label{prop:quenched}
 Let $\pi^{1,L}$ be the  $Q_L-$periodic optimal transport plan for $W_{per}(\mu^{1,L},1)$. In the event $r_{*,L}<L$  we have for every $ r \ge r_{*,L}$,  
\begin{equation}\label{mainweaksto}
 \lt|\int_{\R^d\times \R^d} \eta_r(x)\lt(y-x-\nabla u_r^{1,L}(0)\rt)d\pi^{1,L}\rt|\les  r_{*,L}^2\frac{\beta\lt(\frac{r}{r_{*,L}}\rt)}{r} 
\end{equation}
and 
\begin{multline}\label{strongeststo}
 \sup \lt\{ \lt|x-\lt(y-\nabla u_r^{1,L}(0)\rt)\rt| \ : \ (x,y)\in \spt\pi^{1,L}\cap ((B_{r}\times\R^d)\rt\}\\
 \les  \,r \lt(\frac{r_{*,L}^2 \beta\lt(\frac{r}{r_{*,L}}\rt)}{r^2}\rt)^{\frac{1}{d+2}}.
\end{multline}
If instead $r\le r_{*,L}$,
then for every $\alpha\in(0,1)$,
 \begin{equation}\label{eq:smallscale}
  \lt|\int_{\R^d\times \R^d} \eta_r(x)\lt(y-x-\nabla u^{1,L}_r(0)\rt)d\pi^{1,L}\rt|\les_\alpha r_{*,L} \lt(\frac{r_{*,L}}{r}\rt)^{2d+\alpha}.
 \end{equation}
\end{proposition}
\begin{proof}
 Applying  Theorem \ref{theo:main} to $\mu^{1,L}$, we obtain \eqref{mainweaksto} and \eqref{strongeststo} for $r\ges r_{*,L}$.\\
To simplify notation we let $h_r=\nabla u^{1,L}_r(0)$. By \eqref{eq:concentrationrstar} and \cite[Lem. 2.10]{GHO}, we can find $\bar r\sim r_{*,L}$ large enough for \eqref{mainweaksto} and \eqref{strongeststo} to hold and  such that 
 \begin{equation}\label{eq:difhhyp}
  \frac{1}{|B_{\bar r}|}W_{B_{\bar r}}^2(\mu^{1,L},\kappa)\les  r_{*,L}^2.
 \end{equation}
By Lemma \ref{lem:difh} we have for $r\le \bar r$ and $\alpha\in(0,1)$, 
\begin{equation}\label{eq:difh}
 |h_r-h_{\bar r}|\les_\alpha \lt(\frac{\bar r}{r}\rt)^{d+\alpha} r_{*,L}\les_\alpha \lt(\frac{r_{*,L}}{r}\rt)^{d+\alpha} r_{*,L}.
\end{equation}
Using this for $r=r_{*,L}$ together with the triangle inequality, this extends \eqref{strongeststo} to all $r\ge r_{*,L}$. We now claim that it also implies that  for $r\le \bar r$ and $\alpha\in(0,1)$
\begin{equation}\label{claim:changer}
 \lt|\int_{\R^d\times \R^d} \eta_r(x)\lt(y-x-h_r\rt)d\pi^{1,L}\rt|\les_\alpha r_{*,L} \lt(\frac{r_{*,L}}{R}\rt)^{2d+\alpha}.
\end{equation}
Using this first for $r_{*,L}\le r\le \bar r$ would extend the validity of \eqref{mainweaksto} to all $r\ge r_{*,L}$. The case $r\le r_{*,L}$ would give \eqref{eq:smallscale}. In order to prove \eqref{claim:changer}, we first write
\[
 \int \eta_r(x)\lt(y-x-h_r\rt)d\pi^{1,L}=\int \eta_r(x) \lt(y-x-h_{\bar r}\rt)d\pi^{1,L}+ \int \eta_r(x) \lt(h_{\bar r}-h_r\rt)d\pi^{1,L}
\]
and estimate separately both terms. For the first one we use \eqref{strongeststo}  for $\bar r$ together with the fact that $B_r\subset B_{\bar r}$ and $\bar r\sim r_{*,L}$ to get
\[
 \lt|\int \eta_r(x) (y-x-h^{1,L}_{\bar r})d\pi^{1,L}\rt|\les r_{*,L} \int \eta_r d\mu^{1,L}.
\]
For the second term we use \eqref{eq:difh} and $\bar r\sim r_{*,L}$ to get
\[
 \lt|\int \eta_r(x) (h_{\bar r}-h_r)d\pi^{1,L}\rt|\les_\alpha r_{*,L} \lt(\frac{r_{*,L}}{r}\rt)^{d+\alpha}  \int \eta_r d\mu^{1,L}.
\]
Finally, we can argue
\[
 \int \eta_r d\mu^{1,L}\le \lt(\frac{\bar r}{r}\rt)^d \int \eta_{\bar r} d\mu^{1,L}\les \lt(\frac{r_{*,L}}{r}\rt)^d,
\]
where the second inequality follows from \cite[(4.58) Lem. 4.4]{GHO}. Since $\frac{r_{*,L}}{r}\ges 1$, this ends the proof of \eqref{eq:smallscale}.
\end{proof}

We may now prove the main annealed result of this section: 
\begin{proposition}\label{theop}
 For every $L\gg r\ge 1$ and $p\geq 1$
 \begin{equation}\label{moment}
  \EE\lt[\lt|\int \eta_r (x) (y-x- \nabla u^{1,L}_r(0)) d\pi^{1,L}\rt|^p\rt]^{\frac{1}{p}}\les \frac{\beta(r)}{r}.
 \end{equation}
\end{proposition}
\begin{proof}
 Let $p\ge 1$ be fixed. We start by observing that we have the deterministic bound
 \[
  \lt|\int \eta_r (x) (y-x- \nabla u^{1,L}_r(0)) d\pi^{1,L}\rt|^p\les |\mu_r^{1,L}(0)|^p\lt( L^p  + |\nabla u_r^{1,L}(0)|^p\rt). 
 \]
 Since $\PP[r_{*,L}\geq L]\les \exp \lt(-c \frac{L^2}{\beta(L)}\rt)$ and since we have control on arbitrary high moments of $\mu_r^{1,L}(0)$ and $\nabla u_r^{1,L}(0)$ by Lemma \ref{lem:boundmu} and Lemma \ref{lem:boundnablu}, 
 we may  restrict ourselves to the event $r_{*,L}<L$. 
 We  now estimate
 \begin{align*}
\lefteqn{  \EE\lt[\lt|\int \eta_r (x) \lt(y-x-\nabla u^{1,L}_r(0)\rt) d\pi^{1,L}\rt|^p\rt]^{\frac{1}{p}}}\\
&\le \EE\lt[\chi_{\{r_{*,L}\le r\}} \lt|\int \eta_r (x) (y-x-\nabla u^{1,L}_r(0)) d\pi^{1,L}\rt|^p \rt]^{\frac{1}{p}} \\
  & \quad + \EE\lt[\chi_{\{r_{*,L}\ge r\}} \lt|\int \eta_r (x) (y-x-\nabla u^{1,L}_r(0)) d\pi^{1,L}\rt|^p  \rt]^{\frac{1}{p}} \\
  &\stackrel{\eqref{mainweaksto}\&\eqref{eq:smallscale}}{\les} \frac{1}{r}\EE\lt[r_{*,L}^{2p} \beta^p\lt(\frac{r}{r_{*,L}}\rt)\rt]^{\frac{1}{p}}
  + \EE\lt[\chi_{\{r_{*,L}\ge r\}} \lt(r_{*,L} \lt(\frac{r_{*,L}}{r}\rt)^{2d+\alpha}\rt)^p \rt]^{\frac{1}{p}}. 
 \end{align*}
By the stretched exponential moment of $r_{*,L}$, we have 
\[
\frac{1}{r}\EE\lt[r_{*,L}^{2p} \beta^p\lt(\frac{r}{r_{*,L}}\rt)\rt]^{\frac{1}{p}} + \EE\lt[\chi_{\{r_{*,L}\ge r\}} \lt(r_{*,L} \lt(\frac{r_{*,L}}{r}\rt)^{2d+\alpha}\rt)^p \rt ]^{\frac{1}{p}} \les \frac{\beta(r)}{r},
\]
which concludes the proof.
\end{proof}

\subsection{Main convergence result}\label{sec:mainconv}

We now use Theorem \ref{theo:mainnegsob} and Lemma \ref{lem:removesmallscale} to obtain a quantitative bound on the distance between the displacement and $\nabla u^{R,L}$
in appropriate negative Sobolev spaces. We recall that if $\pi^{R,L}$ is the $Q_L-$periodic optimal transport plan between $\mu^{R,L}$ and $1$, we introduced the distribution $Z^{R,L}$ defined by
\[
 Z^{R,L}(f)=R^{\frac{d}{2}}\int_{Q_L\times \R^d} f(x) (y-x) d\pi^{R,L}.
\]
As already observed in Section \ref{sec:linear} we expect a different behavior for scales larger or smaller than $\frac{1}{R}$.
For the large scales $Z^{R,L}$ will be close to $\nabla u^{R,L}$ and thus  essentially Gaussian 
 while at small scales it has no better regularity than Dirac masses, recall \eqref{Z}.  It is thus natural as in Theorem \ref{theo:convlinearbody} to consider both the behavior of $Z^{R,L}$ and $Z^{R,L}_{\frac{1}{R}}$. We start with the latter since the analysis is simpler.
 \begin{theorem}\label{theo:quantsmall}
  For every $d\ge 3$, $p\ge 2$, $\gamma>0$ with $\gamma\notin \N$, $L\gg\ell\gg1$ and $R\gg1$  we have
  \begin{equation}\label{eq:quantsmall}
  \frac{1}{|B_{\ell}|} \EE[\|Z^{R,L}_{\frac{1}{R}}-\nabla u^{R,L}_{\frac{1}{R}}\|^p_{W^{-\gamma,p}(B_{\ell})}]\les \frac{1}{R^{\frac{p}{2}(4-d)}}\lt(1+R^{p(1-\gamma)}\rt).
  \end{equation}
If $d=2$, $p\ge 2$, $\gamma>0$, with $\gamma\notin \N$,
$L\gg\ell\gg1$ and $R\gg1$  we have
  \begin{multline}\label{eq:quantsmalllog}
  \frac{1}{|B_{\ell}|} \EE\lt[\lt\|Z^{R,L}_{\frac{1}{R}}- \mu^{R,L}_{\frac{1}{R}} \nabla u^{R,L}_{1}(0) -\lt( \nabla u^{R,L}_{\frac{1}{R}}-\nabla u^{R,L}_{1}(0)\rt) \rt\|^p_{W^{-\gamma,p}(B_{\ell})}\rt]\\
  \les\lt(\frac{\log R\ell}{R}\rt)^p\lt(1+R^{p(1-\gamma)}\rt).
  \end{multline}
 \end{theorem}
 \begin{proof}
  We start with \eqref{eq:quantsmall}. We claim that for $\eps\in\lt[\frac{1}{R},1\rt]$ and  $x\in B_{3\ell}$, 
  \begin{equation}\label{claim:quantsmall}
   \EE\lt[\lt|Z_\eps^{R,L}(x)-\nabla u^{R,L}_\eps(x)\rt|^p\rt]^{\frac{1}{p}}\les  \frac{1}{R^{\frac{1}{2}(4-d)}\eps}.
  \end{equation}
Using  \eqref{eq:removesmallscale} of Lemma \ref{lem:removesmallscale} this would yield \eqref{eq:quantsmall}. In order to prove \eqref{claim:quantsmall} we notice
that by stationarity and rescaling (recall \eqref{rescalingnabu}), it is enough to prove that for $L\gg r\ge 1$,
\begin{equation}\label{claim:quantsmallrescale}
   \EE\lt[\lt|Z_r^{1,L}(0)-\nabla u^{1,L}_r(0)\rt|^p\rt]^{\frac{1}{p}}\les  \frac{1}{r}.
  \end{equation}
  Using triangle inequality we have 
  \begin{multline*}
    \EE\lt[\lt|Z_r^{1,L}(0)-\nabla u^{1,L}_r(0)\rt|^p\rt]^{\frac{1}{p}}\le  \EE\lt[\lt|Z_r^{1,L}(0)- \mu_r^{1,L}(0)\nabla u^{1,L}_r(0)\rt|^p\rt]^{\frac{1}{p}}\\
    + \EE\lt[\lt|(1- \mu_r^{1,L}(0))\nabla u^{1,L}_r(0)\rt|^p\rt]^{\frac{1}{p}}.
    \end{multline*}
    For the first term we use 
    \[
    \EE\lt[\lt|Z_r^{1,L}(0)- \mu_r^{1,L}(0)\nabla u^{1,L}_r(0)\rt|^p\rt]^{\frac{1}{p}}=
   \EE\lt[\lt|\int \eta_r(x)(y-x -\nabla u^{1,L}_r(0)) d\pi^{1,L}\rt|^p\rt]^{\frac{1}{p}}
   \stackrel{\eqref{moment}}{\les}\frac{1}{r}.
   \]
   For the second one we use H\"older's inequality to get 
   \begin{multline*}
    \EE\lt[\lt|(1- \mu_r^{1,L}(0))\nabla u^{1,L}_r(0)\rt|^p\rt]^{\frac{1}{p}}\le \EE\lt[\lt|(1- \mu_r^{1,L}(0))\rt|^{2p}\rt]^{\frac{1}{2p}}\EE\lt[\lt|\nabla u^{1,L}_r(0)\rt|^{2p}\rt]^{\frac{1}{2p}}\\
    \stackrel{\eqref{eq:momentmu}\&\eqref{eq:momentnabu}}{\les}  \frac{1}{r^{d-1}}\les \frac{1}{r}.
   \end{multline*}
This concludes the proof of \eqref{claim:quantsmallrescale}.\\
We now turn to  \eqref{eq:quantsmalllog}. We claim that  for $\eps\in\lt[\frac{1}{R},1\rt]$ and  $x\in B_{3\ell}$,
 \begin{equation}\label{claim:quantsmalllog}
   \EE\lt[\lt|Z_\eps^{R,L}(x)-\mu^{R,L}_{\eps}(x) \nabla u^{R,L}_{1}(0) -\lt( \nabla u^{R,L}_{\eps}(x)-\nabla u^{R,L}_{1}(0)\rt)\rt|^p\rt]^{\frac{1}{p}}\les  \frac{\log (R \ell)}{R\eps}.
  \end{equation}
 As above, using \eqref{eq:removesmallscale} of Lemma \ref{lem:removesmallscale} this would give  \eqref{eq:quantsmalllog}.
By scaling, in order to prove \eqref{claim:quantsmalllog}, it is enough to prove that for $1\le r\le R\ll L$ and $x\in Q_L$
\begin{equation}\label{claim:quantsmallrescalelog}
   \EE\lt[\lt|Z_r^{1,L}(x)-\mu^{1,L}_{r}(x) \nabla u^{1,L}_{R}(0) -\lt( \nabla u^{1,L}_{r}(x)-\nabla u^{1,L}_{R}(0)\rt)\rt|^p\rt]^{\frac{1}{p}}\les \frac{\log (|x|+2)+\log R}{r}.
  \end{equation}
For this we write that by triangle inequality,
\begin{multline*}
  \EE\lt[\lt|Z_r^{1,L}(x)-\mu^{1,L}_{r}(x) \nabla u^{1,L}_{R}(0) -\lt( \nabla u^{1,L}_{r}(x)-\nabla u^{1,L}_{R}(0)\rt)\rt|^p\rt]^{\frac{1}{p}}\\
  \le 
  \EE\lt[\lt|Z_r^{1,L}(x)-\mu^{1,L}_{r}(x)\nabla u^{1,L}_r(x) \rt|^p\rt]^{\frac{1}{p}}+  \EE\lt[\lt|(1-\mu^{1,L}_{r}(x)) ( \nabla u^{1,L}_{r}(x)-\nabla u^{1,L}_{R}(0))\rt|^p\rt]^{\frac{1}{p}}.
\end{multline*}
For the first term we use stationarity to infer 
\begin{align*}
 \EE\lt[\lt|Z_r^{1,L}(x)-\mu^{1,L}_{r}(x)\nabla u^{1,L}_r(x) \rt|^p\rt]^{\frac{1}{p}}&=\EE\lt[\lt|Z_r^{1,L}(0)-\mu^{1,L}_{r}(0)\nabla u^{1,L}_r(0) \rt|^p\rt]^{\frac{1}{p}}\\
 &=\EE\lt[\lt|\int \eta_r(x)(y-x -\nabla u^{1,L}_r(0)) d\pi^{1,L} \rt|^p\rt]^{\frac{1}{p}}\\
 &\stackrel{\eqref{moment}}{\les} \frac{\log(r+1)}{r}.
\end{align*}
For the second term we use as above, H\"older's inequality in combination with \eqref{eq:momentmu} and this time \eqref{eq:momentnabulog} to obtain 
\begin{multline*}
 \EE\lt[\lt|(1-\mu^{1,L}_{r}(x)) ( \nabla u^{1,L}_{r}(x)-\nabla u^{1,L}_{R}(0))\rt|^p\rt]^{\frac{1}{p}}\\
 \les \frac{1}{r} \lt(\log^{\frac{1}{2}}\lt(\frac{R+|x|}{r}\rt)+1\rt) \les 
 \frac{\log\lt(R+|x|+1\rt)}{r}.
\end{multline*}
\end{proof}
\begin{remark}\label{rem:hom}
 If we assume that $R$ is a not too slowly diverging sequence in terms of $L$ (more precisely if $R\gg \log^{\frac{1}{2\gamma}} L$), it can be proven using the same kind of computations that 
 $
  \EE\lt[\| (\mu_{\frac{1}{R}}^{R,L}-1)\nabla u_1^{R,L}(0)\|_{W^{-\gamma,p}(B_\ell)}^p\rt]
 $
is small and thus replace \eqref{eq:quantsmalllog} by an estimate on the simpler quantity
$
 \EE\lt[\|Z^{R,L}_{\frac{1}{R}}- \nabla u^{R,L}_{\frac{1}{R}}\|_{W^{-\gamma,p}(B_\ell)}^p\rt].
$
\end{remark}
We finally consider also the small scales $\eps\le \frac{1}{R}$ (the analog of Remark \ref{rem:hom} of course also holds).
\begin{theorem}\label{theo:quantall}
  For every $d\ge 3$, $p\ge 2$, $\gamma> d\lt(1-\frac{1}{p}\rt)$ with $\gamma\notin \N$, $L\gg\ell\gg1$ and $R\gg1$  we have
  \begin{equation}\label{eq:quantall}
  \frac{1}{|B_{\ell}|} \EE[\|Z^{R,L}-\nabla u^{R,L}\|^p_{W^{-\gamma,p}(B_{\ell})}]\les \frac{1}{R^{\frac{p}{2}(4-d)}}.
  \end{equation}
If $d=2$, $p\ge 2$, $\gamma>2\lt(1-\frac{1}{p}\rt)$ with $\gamma\notin \N$,
$L\gg\ell\gg1$ and $R\gg1$  we have
  \begin{equation}\label{eq:quantalllog}
  \frac{1}{|B_{\ell}|} \EE\lt[\lt\|Z^{R,L}- \mu^{R,L} \nabla u^{R,L}_{1}(0) -\lt( \nabla u^{R,L}-\nabla u^{R,L}_{1}(0)\rt) \rt\|^p_{W^{-\gamma,p}(B_{\ell})}\rt]\les\lt(\frac{\log R\ell}{R}\rt)^p.
  \end{equation}
\end{theorem}
\begin{proof}
Within this proof we use the shorthand notation $A\les f(p_+)$ if for any $1\gg \delta>0$, $A\les_\delta f((1+\delta)p)$.

 We start with \eqref{eq:quantall}. In light of \eqref{eq:main} of Theorem \ref{theo:mainnegsob} and \eqref{claim:quantsmall}, it is enough to prove that for $\eps\in(0,\frac{1}{R})$ and $x\in B_{2\ell}$,
 \begin{equation}\label{claim:quantall}
   \EE\lt[\lt|Z_\eps^{R,L}(x)-\nabla u^{R,L}_\eps(x)\rt|^p\rt]^{\frac{1}{p}}\les   \frac{1}{(\eps R)^{1+\frac{d(p_+-2)}{2p_+}}}\frac{1}{ \eps^{\frac{1}{2}(d-2)}}.
  \end{equation}
  Indeed, this would give for $\gamma> d\lt(1-\frac{1}{p}\rt)>1$
  \begin{multline*}
  \frac{1}{|B_{\ell}|} \EE\lt[ \int_0^{\frac{1}{R}}\int_{B_{2\ell}}  \eps^{p\gamma}\lt|Z_\eps^{R,L}(x)-\nabla u^{R,L}_\eps(x)\rt|^p dx \frac{d\eps}{\eps}\rt]\les 
  \int_0^{\frac{1}{R}} \eps^{p\gamma} \frac{1}{(\eps R)^{p+\frac{p}{p_+}\frac{d(p_+-2)}{2}}}\frac{1}{ \eps^{\frac{p}{2}(d-2)}} \frac{d\eps}{\eps}\\
  \les \frac{1}{R^{p\lt(\gamma-\frac{d-2}{2}\rt)}}\les  \frac{1}{R^{\frac{p}{2}(4-d)}}.
  \end{multline*}

  We thus prove \eqref{claim:quantall}. Since we are considering scales for which the linearization ansatz is not expected to be relevant, we use triangle inequality, \eqref{eq:momentnabu}, stationarity and rescaling to reduce the estimate to show for $0<r<1$
  \begin{equation}\label{claim:quantallreduced}
   \EE\lt[\lt|\int \eta_r(x)(y-x) d\pi^{1,L}\rt|^p\rt]^{\frac{1}{p}}\les \frac{1}{r^{\frac{d}{2}}} \frac{1}{r^{\frac{d(p_+-2)}{2p_+}}}.
  \end{equation}
Let $q=(1+\delta)p$ with $0<\delta\ll1$. Recalling the definition of the random radius $r_{*,L}=r_*\geq 1$ (dropping the $L$ for the rest of the proof)  from Section \ref{sec:sto}, notice first that thanks
to its stretched exponential moments, the deterministic bound $|x-y|\les L$ for $(x,y)\in \spt \pi^{1,L}$ and the moment bound \eqref{eq:momentmu} for $\mu_r^{1,L}(0)$, 
we may reduce ourselves to the event $r_*<L$. In this event, we use the triangle inequality to write 
\begin{multline*}
  \EE\lt[\lt|\int \eta_r(x)(y-x) d\pi^{1,L}\rt|^p\rt]^{\frac{1}{p}}\les  \EE\lt[\lt|\int \eta_r(x)(y-x-\nabla u_{r_*}^{1,L}(0)) d\pi^{1,L}\rt|^p\rt]^{\frac{1}{p}}
  \\+ \EE\lt[\lt|\mu_r^{1,L}(0)(\nabla u_{r_*}^{1,L}(0)-\nabla u_1^{1,L}(0))\rt|^p\rt]^{\frac{1}{p}} + \EE\lt[\lt|\mu_r^{1,L}(0)\nabla u_1^{1,L}(0)\rt|^p\rt]^{\frac{1}{p}}.
\end{multline*}
For the first term we use \eqref{strongeststo} for $r=r_*$ together with the fact that $B_r\subset B_{r_*}$ (since $r\leq 1 \leq r_*$) to obtain
\begin{align*}
  \EE\lt[\lt|\int \eta_r(x)(y-x-\nabla u_{r_*}^{1,L}(0)) d\pi^{1,L}\rt|^p\rt]^{\frac{1}{p}}&\les  \EE\lt[\lt| \mu_r^{1,L}(0) r_* \rt|^p\rt]^{\frac{1}{p}}\\
 & \stackrel{\textrm{H\"older}}{\le}  \EE\lt[\lt| \mu_r^{1,L}(0)\rt|^{q}\rt]^{\frac{1}{q}} \EE\lt[ r_*^{\frac{p(1+\delta)}{\delta}}\rt]^{\frac{\delta}{p (1+\delta)}}\\
  &\stackrel{\eqref{eq:momentmu}}{\les_\delta} \frac{1}{r^{\frac{d}{2}}} \frac{1}{r^{\frac{d(q-2)}{2q}}}.
\end{align*}
For the second term we use\footnote{Up to adding yet another intermediate radius $r'\sim r_*$ and using triangle inequality we may assume that \eqref{eq:difhhyp} holds for $\bar r=r_*$.}
\eqref{eq:difh} to obtain
\[
 \EE\lt[\lt|\mu_r^{1,L}(0)(\nabla u_{r_*}^{1,L}(0)-\nabla u_1^{1,L}(0))\rt|^p\rt]^{\frac{1}{p}}\les_\alpha  \EE\lt[\lt|\mu_r^{1,L}(0){r_{*}^{d+\alpha+1}}\rt|^p\rt]^{\frac{1}{p}}\stackrel{\eqref{eq:momentmu}}{\les_\delta} \frac{1}{r^{\frac{d}{2}}} \frac{1}{r^{\frac{d(q-2)}{2q}}}.
\]
Finally the last term is estimated thanks to H\"older's inequality as 
\[
 \EE\lt[\lt|\mu_r^{1,L}(0)\nabla u_1^{1,L}(0)\rt|^p\rt]^{\frac{1}{p}}
 \le \EE\lt[\lt| \mu_r^{1,L}(0)\rt|^{q}\rt]^{\frac{1}{q}} \EE\lt[ |\nabla u_1^{1,L}(0)|^{\frac{p(1+\delta)}{\delta}}\rt]^{\frac{\delta}{p (1+\delta)}}\stackrel{\eqref{eq:momentmu}\&\eqref{eq:momentnabu}}{\les_\delta} \frac{1}{r^{\frac{d}{2}}} \frac{1}{r^{\frac{d(q-2)}{2q}}}.
\]
This proves \eqref{claim:quantallreduced}.\\

We finally show \eqref{eq:quantalllog}. As above, from \eqref{eq:main} and \eqref{claim:quantsmalllog}, it is enough to prove that for $\eps\in(0,\frac{1}{R})$, { $x\in B_\ell$}
\begin{equation}\label{claim:quantalllog}
   \EE\lt[\lt|Z_\eps^{R,L}(x)-\mu^{R,L}_{\eps}(x) \nabla u^{R,L}_{1}(0) -\lt( \nabla u^{R,L}_{\eps}(x)-\nabla u^{R,L}_{1}(0)\rt)\rt|^p\rt]^{\frac{1}{p}}\les\frac{1}{(\eps R)^{1+\frac{p_+-2}{p_+}}}  \log^{\frac{1}{2}} (R \ell).
  \end{equation}
  Using the triangle inequality, \eqref{eq:momentnabulog} and rescaling, we are left with the proof of 
  \begin{equation}\label{claim:quantalllogreduced}
   \EE\lt[\lt|\int \eta_r (y-x) d\pi^{1,L} -\mu^{1,L}_{r}(0) \nabla u^{1,L}_{R}(0)\rt|^p\rt]^{\frac{1}{p}}\les \frac{1}{r}\frac{1}{r^{\frac{p_+-2}{p_+}}}  \log^{\frac{1}{2}} (R \ell)
  \end{equation}
Fix $q=(1+\delta)p$ with $0<\delta\ll1$. As before, we may reduce ourselves to the event $r_*<L$. We use triangle inequality to write
\begin{align*}
\lefteqn{ \EE\lt[\lt|\int \eta_r (y-x) d\pi^{1,L} -\mu^{1,L}_{r}(0) \nabla u^{1,L}_{R}(0)\rt|^p\rt]^{\frac{1}{p}}}\\
&\le
 \EE\lt[\lt|\int \eta_r (y-x - \nabla u_{r_*}^{1,L}(0)) d\pi^{1,L}\rt|^p\rt]^{\frac{1}{p}}+ \EE\lt[\lt| \mu_r^{1,L}(0) ( \nabla u_{r_*}^{1,L}(0)-\nabla u_1^{1,L}(0))\rt|^p\rt]^{\frac{1}{p}} \\
 &+ \EE\lt[\lt| \mu_r^{1,L}(0) (\nabla u_1^{1,L}(0)-\nabla u_R^{1,L}(0))\rt|^p\rt]^{\frac{1}{p}}.
\end{align*}
The first term is bounded exactly as before using \eqref{strongeststo} and \eqref{eq:momentmu}:
\[
  \EE\lt[\lt|\int \eta_r (y-x - \nabla u_{r_*}^{1,L}(0)) d\pi^{1,L}\rt|^p\rt]^{\frac{1}{p}}\les \EE\lt[\lt| \mu_r^{1,L}(0) r_*\rt|^p\rt]^{\frac{1}{p}}\les_\delta \frac{1}{r}\frac{1}{r^{\frac{q-2}{q}}}.
\]
As above, we obtain a similar upper bound for the second term  using \eqref{eq:difh}. For the third term we use H\"older's inequality and \eqref{eq:momentnabulog} to get
\begin{multline*}
 \EE\lt[\lt| \mu_r^{1,L}(0) (\nabla u_1^{1,L}(0)-\nabla u_R^{1,L}(0))\rt|^p\rt]^{\frac{1}{p}}\\
 \les 
 \EE\lt[\lt| \mu_r^{1,L}(0) \rt|^q\rt]^{\frac{1}{q}} \EE\lt[\lt|\nabla u_1^{1,L}(0)-\nabla u_R^{1,L}(0)\rt|^{\frac{p(1+\delta)}{\delta}}\rt]^{\frac{\delta}{p (1+\delta)}}\\
 \les_\delta \frac{1}{r}\frac{1}{r^{\frac{q-2}{q}}}  \log^{\frac{1}{2}} R.
\end{multline*}
This concludes the proof of \eqref{claim:quantalllogreduced}.
\end{proof}

Combining the estimates from Theorem \ref{theo:convlinearbody}, Theorem \ref{theo:quantsmall}, and Theorem \ref{theo:quantall} we obtain our main result including Theorems \ref{thm:intro1} and  \ref{thm:intro2}.

\begin{theorem}\label{thm:maincvg}
A) Let $p\geq 2$, $\gamma > d\lt(1-\frac1p\rt)$.  If $d=3$, $Z^{R,L}$ converges is law in $W^{-\gamma,p}_{loc}$ to $\nabla\Psi$ as $R,L\to\infty$. If $d=2$, $Z^{R,L}-\mu^{R,L}\nabla u^{R,L}_1(0)$ converges in law in $W_{loc}^{-\gamma,p}$ to $\nabla\Psi-\nabla\Psi_1(0)$ as $R,L\to\infty$.

B) Let $p\geq 2$, $\gamma > \frac{d}{2}-1$. If $d=3$, $Z_\frac{1}{R}^{R,L}$ converges in law in $W_{loc}^{-\gamma,p}$ to $\nabla\Psi$ as $R,L\to\infty$. If $d=2$, $Z_\frac{1}{R}^{R,L}-\mu_\frac{1}{R}^{R,L}\nabla u^{R,L}_1(0)$ converges in law in $W_{loc}^{-\gamma,p}$ to $\nabla\Psi-\nabla\Psi_1(0)$ as $R,L\to\infty$.\\

Moreover, under these assumptions, for every $\ell\ge 1$, these random distributions have bounded moments of arbitrary order in $W^{-\gamma,p}(B_\ell)$.
\end{theorem}

 \bibliographystyle{abbrv}

\bibliography{OT}
\end{document}